\documentclass[9pt]{article}
\usepackage{amsmath}                        
\usepackage{amsthm}                         
\usepackage{amssymb}                        
\usepackage{amsfonts}                       
\usepackage{latexsym}                       
\usepackage{setspace}						
\usepackage[colorlinks=true,linkcolor=blue,citecolor=blue]{hyperref}
\usepackage[all]{xy}
\usepackage{parskip}

\newtheorem{lem}{Lemma}[section]
\newtheorem{thm}[lem]{Theorem}

\newtheorem{prop}[lem]{Proposition}
\newtheorem{coro}[lem]{Corollary}
\newtheorem{remar}[lem]{Remark}
\newcommand{\e}{\it 1}
\newcommand{\R}{\mathbb{R}}

\newcommand{\C}{\mathbb{C}}
\newcommand{\Q}{\mathcal{Q}}

\newcommand{\HC}{\mathbb{H}^{\mathbb{C}}}

\newcommand{\la}{\langle}
\newcommand{\ra}{\rangle}
\newcommand{\laa}{\langle\langle}
\newcommand{\raa}{\rangle\rangle}
\title{\bf Timelike surfaces into $4-$dimensional Minkowski space via spinors}
\author{Victor H. Patty-Yujra \\ victorp@im.unam.mx \\ Instituto de Matem\'aticas UNAM, Juriquilla-Quer\'etaro, M\'exico} 

\begin{document}
\maketitle

\begin{abstract}
We prove that an isometric immersion of a timelike surface in four-dimensional Minkowski space is equivalent to a normalized spinor field which is a solution of a Dirac equation on the surface. Using the quaternions and the complex numbers, we obtain a spinor representation formula that relates the spinor field and the isometric immersion. Applying the representation formula, we deduce a new spinor representation of a timelike surface in three-dimensional De Sitter space; we give a formula for the Laplacian of the Gauss map of a minimal timelike surface in four-dimensional Minkowski space in terms of the curvatures of the surface;  we obtain a local description of a flat timelike surface with flat normal bundle and regular Gauss map in four-dimensional Minkowski space, and we also give a conformal description of a flat timelike surface in three-dimensional De Sitter space.
\end{abstract}



\section{Introduction}
We consider $\R^{3,1}$ the four-dimensional Minkowski space defined by $\R^4$ endowed with indefinite metric of signature $(3,1)$ given by $$\la\cdot,\cdot\ra=-dx_1^2+dx_2^2+dx_3^2+dx_4^2.$$ 
A surface $M\subset \R^{3,1}$ is said to be timelike if the metric $\la\cdot,\cdot\ra$ induces on $M$ a metric of signature $(1,1).$ In this paper, we are interested in the spinorial description of a timelike surface in $\R^{3,1},$ with given normal bundle and given mean curvature vector, and in its applications to the geometry of timelike surfaces in $\R^{3,1}.$ With this, we pretend to complete the spinorial description of semi-Riemannian surfaces in four-dimensional semi-Riemannian Euclidean spaces \cite{bayard1,bayard_lawn_roth,bayard_patty}.

Below we will state the main result of this paper. Let $M$ be an abstract simply connected timelike surface, $E\to M$ be a bundle of rank $2$ with a Riemannian metric and a compatible connection. We assume moreover that spin structures are given on $TM$ and on $E,$ and we define $\Sigma:=\Sigma M\otimes\Sigma E,$ the tensor product of the corresponding bundles of spinors. Let $\HC$ be the space of quaternions with coefficients in $\C$ defined by \[\HC:=\{ q_1\e+q_2I+q_3J+q_4K\mid q_1,q_2,q_3,q_4 \in \C\},\] where $I,J$ and $K$ are such that \[ I^2=J^2=K^2=-1 \hspace{0.2in}\mbox{and}\hspace{0.2in} IJ=-JI=K.\] 
We will see (Section \ref{preliminares}) that two natural bilinear maps
\begin{equation*}
H:\Sigma \times \Sigma \to \C \hspace{0.2in}\mbox{and}\hspace{0.2in} \laa\cdot,\cdot\raa: \Sigma\times\Sigma \to \HC
\end{equation*} are defined on $\Sigma.$ We have the following: \\

\begin{thm}\label{principal}
Let $\vec{H}$ be a section of $E.$ The following three statements are equivalent.
\begin{enumerate}
\item[1-] There exists a spinor field $\varphi\in\Gamma(\Sigma)$ with $H(\varphi,\varphi)=1$ solution of the Dirac equation
\begin{equation*}
D\varphi=\vec{H}\cdot\varphi.
\end{equation*}
\item[2-] There exists a spinor field $\varphi\in\Gamma(\Sigma)$ with $H(\varphi,\varphi)=1$ solution of 
\begin{equation*}
\nabla_X\varphi=-\frac{1}{2}\sum_{j=1,2}\epsilon_je_j\cdot B(X,e_j)\cdot\varphi,
\end{equation*}
where $B:TM\times TM\to E$ is a bilinear and symetric map with $\frac{1}{2}tr_{}B=\vec{H},$ and where $(e_1,e_2)$ is an orthonormal frame of $TM$ and  $\epsilon_j=\la e_j,e_j\ra.$
\item[3-] There exists an isometric immersion $F:M \to \R^{3,1}$ with normal bundle $E,$ second fundamental form $B$ and mean curvature vector $\vec{H}.$
\end{enumerate} 
Moreover, the isometric immersion is given by the spinor representation formula $$F=\int\xi:M \longrightarrow \R^{3,1}\hspace{0.2in}\mbox{with}\hspace{0.2in} \xi(X):=\laa X\cdot\varphi,\varphi\raa,$$  for all $X\in TM,$ where $\xi$ is a closed $1$-form on $M$ with values in $\R^{3,1}.$
\end{thm}

The definitions of the Clifford product $"\cdot"$ on the spinor bundle $\Sigma,$ of the Dirac operator $D$ acting on $\Gamma(\Sigma)$ and of the immersion of $\R^{3,1}$ into $\HC$ are given in Section \ref{preliminares}. The proof of this theorem will be given in Section \ref{spinor rep}.

Using the representation formula, we give various applications concerning to the geometry of timelike surfaces in $\R^{3,1}:$ we start with a spinorial proof of the fundamental theorem of submanifolds (Remark \ref{teo_fund_inmersion}, Corollary \ref{tfs}); see in \cite{bayard1,bayard_lawn_roth,bayard_patty} a similar application in other contexts. We also give a classical formula for the Laplacian of the isometric immersion (Corollary \ref{laplacian-immersion}). 

As a second application of Theorem \ref{principal}, we give (using only one intrinsic spinor) a new representation of a timelike surface in three-dimensional De Sitter space (Remark \ref{S21}); our representation is different to that given in \cite{lawn_roth} where two spinors are needed. We also recover the representation of a timelike surface in three-dimensional Minkowski space given in \cite{bayard_patty}.  

The thrid application of Theorem \ref{principal} is a formula for the Laplacian of the Gauss map of a minimal timelike surface in $\R^{3,1}$ in terms of the Gauss and normal curvatures of the surface (Corollary \ref{laplacian-gauss-minima}); this formula generalizes a classical formula for minimal surfaces in Euclidean space.

The fourth application of Theorem \ref{principal} is the local description of a flat timelike surface with flat normal bundle and regular Gauss map in $\R^{3,1}$ (Corollary \ref{flat0} and \ref{flat1}). Using the extrinsic geometry of the immersion, we  give to the surface a Riemann surface structure with respect to which its Gauss map is holomorphic, and we prove that these surfaces are described by two holomorphic functions and two smooth functions satisfying a condition of compatibility; this is the main result of \cite{AGM}, that we prove here using spinors. 

The last application obtained from Theorem \ref{principal} is a conformal description of a flat timelike surface in three-dimensional De Sitter space (Corollary \ref{flat3}); our representation coincides with the description given by Aledo, G\'alvez and Mira in \cite[Corollary 5.1]{AGM}.

We quote the following related papers: the spinor representation of surfaces in $\R^3$ was studied by many authors, especially by Friedrich \cite{friedrich}, who interpreted a spinor field representing a surface in $\R^3$ as a constant spinor field of $\R^3$ restricted to the surface; following this approach, surfaces in $\mathbb{S}^3$ and $\mathbb{H}^3$ was studied by Morel \cite{morel} and surfaces in three-dimensional semi-Riemannian space forms was studied by Lawn and Roth \cite{lawn,lawn_roth}. The last two authors together with Bayard studied in \cite{bayard_lawn_roth} surfaces in four-dimensional space forms; spacelike surfaces in $\R^{3,1}$ was studied by Bayard \cite{bayard1}. The author together with Bayard studied in \cite{bayard_patty} Lorentzian surfaces in $\R^{2,2}$ and, different applications of this representation were given by the author in \cite{patty}. 

The paper is organized as follows. In Section \ref{preliminares} we describe the preliminaries concerning the spinors of $\R^{3,1}$ and the spin geometry of a timelike surface in $\R^{3,1}.$ In Section \ref{spinor rep} we prove the spinor representation theorem and we also give the spinor representation formula of the immersion by the spinor field. In Section \ref{app-1} we study the isometric immersion of a timelike surface in three-dimensional De Sitter space. Section \ref{app-2} is devoted to compute the Laplacian of the Gauss map of a timelike surface in $\R^{3,1}.$ We obtain the local description of a flat timelike surface with flat normal bundle and regular Gauss map in $\R^{3,1}$ in Section \ref{app-3}. Finally, in Section \ref{app-4} we deduce a conformal description of a flat timelike surface in three-dimensional De Sitter space.

\section{Preliminaries}\label{preliminares}
\subsection{Spinors of $\R^{3,1}$}\label{spinors R31}
In this section we describe the Clifford algebra of $\R^{3,1},$ the spinorial group and their representations (see \cite[Section 1]{bayard1}).

Using the Clifford map
\begin{eqnarray*}\label{aplicliff}
\R^{3,1} & \longrightarrow & \HC(2) \\
(x_1,x_2,x_3,x_4) & \longmapsto & 
\begin{pmatrix} 0 & i x_1\e+x_2I+x_3J+x_4K\\ -i x_1\e+x_2I+x_3J+x_4K & 0 \end{pmatrix}\notag
\end{eqnarray*} where $\HC(2)$ stands for the set of $2\times2$ matrices with entries belonging to $\HC,$ we get the Clifford algebra of $\R^{3,1}$ 
\begin{equation*} Cl(3,1)=\left\lbrace\begin{pmatrix} p & q\\ \widehat{q} & \widehat{p} \end{pmatrix}\in \HC(2)\mid p, q\in \HC \right\rbrace,
\end{equation*} where $\widehat{q}:=\widehat{q}_11+\widehat{q}_2I+\widehat{q}_3J+\widehat{q}_4K$ ($\widehat{q}_j$ means the usual conjugation in $\C$ of $q_j$) for all $q=q_1\e+q_2I+q_3J+q_4K\in \HC.$ 
The Clifford sub-algebra of elements of even degree is   
\begin{align}\label{elempar}
Cl_0(3,1)& =\left\lbrace \begin{pmatrix}
p & 0\\
0 & \widehat{p}
\end{pmatrix}\in \HC(2)\mid p\in\HC\right\rbrace\simeq\HC
\end{align}
and the subspace of elements of odd degree is
\begin{align*}\label{elem_impar}
Cl_1(3,1)& =\left\lbrace \begin{pmatrix}
0 & q\\
\widehat{q} & 0
\end{pmatrix}\in \HC(2)\mid q\in\HC\right\rbrace\simeq\HC.
\end{align*}

We consider the map $H:\HC\times\HC \longrightarrow \C$ defined by
\begin{equation*}
H(p,p')=p_1p_1'+p_2p_2'+p_3p_3'+p_4p_4'
\end{equation*}
where $p=p_1\e+p_2I+p_3J+p_4K$ and $p'=p_1'\e+p_2'I+p_3'J+p_4'K.$ It is $\C$-bilinear and symmetric. Its real part, denoted by $\la\cdot,\cdot\ra,$ is a real scalar product of signature $(4,4)$ on $\HC.$ The spinorial group is given by
\begin{equation*}
Spin(3,1):=\left\lbrace p\in \HC \mid H(p,p)=1 \right\rbrace\subset Cl_0(3,1).
\end{equation*}

Now, if we consider the identification 
\begin{equation}\label{iden_espa}
\R^{3,1}  \simeq \{i x_1\e+x_2I+x_3J+x_4K \in \HC\mid x_j\in \R\} 
\simeq \{q \in \HC\mid q=-\widehat{\overline{q}}\}, 
\end{equation}
where, if $q=q_1\e+q_2I+q_3J+q_4K\in \HC,$ $\overline{q}:=q_1\e-q_2I-q_3J-q_4K$ is the usual conjugation in $\HC,$ we get the double cover  
\begin{eqnarray}\label{cubriente}
\Phi: & Spin(3,1) & \longrightarrow  SO(3,1)\\
& p & \longmapsto  (q\in\R^{3,1} \longmapsto p\ q\ \widehat{p}^{-1}\in\R^{3,1}). \notag
\end{eqnarray} 
Here and below $SO(3,1)$ stands for the component of the identity of the semi-orthogonal group $O(3,1)$ (see \cite{oneill}). 

Let us denote by $\rho: Cl(3,1) \longrightarrow End_{\C}(\HC)$ the complex representation of $Cl(3,1)$ on $\HC$ given by 
\begin{equation*}\label{rep_alg} \rho\begin{pmatrix}p & q\\ \widehat{q} & \widehat{p}\end{pmatrix}:\hspace{0.2in} \xi\simeq \begin{pmatrix}\xi \\ \widehat{\xi}\end{pmatrix}\longmapsto \begin{pmatrix}p & q\\ \widehat{q} & \widehat{p}\end{pmatrix}\begin{pmatrix}\xi \\ \widehat{\xi}\end{pmatrix}\simeq p\xi+q\widehat{\xi},
\end{equation*} 
where the complex structure on $\HC$ is given by the multiplication by $K$ on the right. The spinorial representation of $Spin(3,1)$ is the restriction to $Spin(3,1)$ of the representation $\rho$ and reads 
\begin{eqnarray*}\label{rep_spin}
\rho_{|Spin(3,1)}: Spin(3,1) & \longrightarrow & End_{\C}(\HC) \\
p & \longmapsto & (\xi\in \HC \longmapsto p\xi\in\HC). \notag
\end{eqnarray*}
This representation splits into 
$\HC= S^+ \oplus S^-,$ where $S^+=\{ \xi\in\HC\mid \xi K=i\xi \}$ and $S^-=\{ \xi\in\HC\mid \xi K=-i\xi\};$ explicitly we have $$S^+=(\C\oplus\C J)(\e-iK)\hspace{0.2in}\mbox{and}\hspace{0.2in} S^-=(\C\oplus\C J)(\e+iK).$$ Note that, if $(e_1,e_2,e_3,e_4)$ stands for the canonical basis of $\R^{3,1},$ the complexified volume element $i\ e_1\cdot e_2\cdot e_3\cdot e_4$ acts as $+Id$ on $S^+$ and as $-Id$ on $S^-.$

\textit{Spinors under the splitting $\R^{3,1}=\R^{1,1}\times \R^2.$}
We consider the splitting $\R^{3,1}=\R^{1,1}\times\R^2$ and the corresponding inclusion $SO(1,1)\times SO(2)\subset SO(3,1).$ Using the definition \eqref{cubriente} of $\Phi,$ we get 
\begin{equation*}
\Phi^{-1}(SO(1,1)\times SO(2))=\{ \cos z+ \sin z I \mid z\in\C \}=:S^1_{\C} \subset Spin(3,1);
\end{equation*} more precisely, setting $z=r+is,$ $r,s\in\R,$ we have in $\HC,$ \begin{equation*}
\cos z+\sin z I=(\cosh s+i\sinh s I)(\cos r+\sin r I ),
\end{equation*} and $\Phi(\cos z+\sin z I)$ is the Lorentz transformation of $\R^{3,1}$ which consists of a Lorentz transformation of angle $2s$ in $\R^{1,1}$ and a rotation of angle $2r$ in $\R^2.$ Thus, defining 
\begin{equation*}\label{spin11}
Spin(1,1):=\{ \pm(\cosh s+i\sinh s I) \mid s\in\R \} \subset Spin(3,1)
\end{equation*} and 
\begin{equation*}\label{spin22}
Spin(2):= \{ \cos r+\sin r I \mid r\in\R \} \subset Spin(3,1),
\end{equation*} we have \begin{equation*}
S^1_{\C}=Spin(1,1).Spin(2)\simeq Spin(1,1)\times Spin(2) / \mathbb{Z}_2
\end{equation*} and the double cover $\Phi: S^1_{\C} \longrightarrow SO(1,1)\times SO(2).$

Finally, the representation 
\begin{eqnarray*}
Spin(1,1)\times Spin(2) & \longrightarrow & End_{\C}(\HC) \\
(g_1,g_2) & \longmapsto & \rho(g):\xi \to g\xi
\end{eqnarray*} where $g=g_1g_2\in S^1_{\C}=Spin(1,1).Spin(2),$ is equivalent to the representation $\rho_1\otimes\rho_2$ of $Spin(1,1)\times Spin(2),$ where $\rho_1$ and $\rho_2$ are the spinorial representations of $Spin(1,1)$ and $Spin(2);$ see \cite[Remark 1.1]{bayard1}.

\subsection{Spin geometry of a timelike surface in $\R^{3,1}$}\label{gauss_formula} 
\textit{Fundamental equations.} Let $M$ be an oriented timelike surface in $\R^{3,1}$ with normal bundle $E$ and second fundamental form $B:TM\times TM\to E$ defined by \[B(X,Y)=\overline{\nabla}_XY-\nabla_XY,\] where $\nabla$ and $\overline{\nabla}$ are the Levi-Civita connections of $M$ and $\R^{3,1}$ respectively. The second fundamental form satisfies the following fundamental equations (\cite{oneill}):
\begin{enumerate}
\item[1-] $K=|B(e_1,e_2)|^2-\la B(e_1,e_1),B(e_2,e_2) \ra$ (Gauss equation),
\item[2-] $K_N=\la (S_{e_3}\circ S_{e_4}-S_{e_4}\circ S_{e_3})(e_1),e_2\ra$ (Ricci equation),
\item[3-] $(\tilde{\nabla}_XB)(Y,Z)-(\tilde{\nabla}_YB)(X,Z)=0$ (Codazzi equation),
\end{enumerate}
where $K$ and $K_N$ are the curvatures of $M$ and $E,$ $(e_1,e_2)$ and $(e_3,e_4)$ are orthonormal basis of $TM$ and $E$ respectively, and where $\tilde{\nabla}$ is the natural connection induced on $T^*M^{\otimes 2}\otimes E.$ As usual, if $\nu\in E,$ $S_{\nu}$ stands for the symmetric operator on $TM$ such that, for all $X,Y\in TM,$ $$\la S_{\nu}(X),Y\ra=\la B(X,Y),\nu\ra.$$  

\begin{remar}\label{teo_fund_inmersion}
Let $M$ be an abstract timelike surface, $E\to M$ be a bundle of rank $2,$ equipped with a Riemannian metric and a compatible connection. We assume that $B:TM\times TM\to E$ is a bilinear map satisfying the equations $\textit{1-, 2-}$ and $\textit{3-}$ above; the fundamental theorem of submanifolds says that there exists locally a unique isometric immersion of $M$ in $\R^{3,1}$ with normal bundle $E$ and second fundamental form $B.$ We will prove this theorem in Corollary \ref{tfs}.
\end{remar}

\textit{Spinorial Gauss formula.} There exists an identification between the spinor bundle of $\R^{3,1}$ over $M,$  $\Sigma\R^{3,1}_{|M},$ and the spinor bundle of $M$ twisted by the spinorial normal bundle, $\Sigma:=\Sigma M\otimes\Sigma E$ (see \cite{bar} and the end of Section \ref{spinors R31}). Moreover, as in the Riemannian case we obtain a spinorial Gauss formula: for all $\varphi\in\Gamma(\Sigma)$ and all $X\in TM,$ 
\begin{equation*}\label{formula_gauss_espinorial}
\overline{\nabla}_X\varphi=\nabla_X\varphi+\frac{1}{2}\sum_{j=1,2}\epsilon_je_j\cdot B(X,e_j)\cdot\varphi.
\end{equation*}
where $\epsilon_j=\la e_j,e_j\ra,$ $\overline{\nabla}$ is the spinorial connection of $\Sigma\R^{3,1},$ $\nabla$ is the spinorial connection of $\Sigma$ defined by $\nabla=\nabla^M\otimes\nabla^E$ the tensor product of the spinor connections on $\Sigma M$ and on $\Sigma E,$ 
and the dot $"\cdot"$ is the Clifford acction of $\R^{3,1}.$ Thus, if we take $\varphi\in\Sigma\R^{3,1}$ parallel, its restriction to $M$, $\varphi:=\varphi|_M$ satisfies
\begin{equation*}\label{for_gauss_paralelo}
\nabla_X\varphi=-\frac{1}{2}\sum_{j=1,2}\epsilon_je_j\cdot B(X,e_j)\cdot\varphi,
\end{equation*} for all $X\in TM.$ Taking the trace, we have the following Dirac equation
\begin{equation*}\label{ecua_dirac_inmersion}
D\varphi=\vec{H}\cdot \varphi,
\end{equation*}
where $D\varphi:=-e_1\cdot\nabla_{e_1}\varphi+e_2\cdot\nabla_{e_2}\varphi$ and where $\vec{H}=\frac{1}{2}tr_{\la,\ra}B$ is the mean curvature vector of $M$ in $\R^{3,1}.$ 

\subsection{Twisted spinor bunble}\label{twisted sec}
Let $M$ be an abstract oriented timelike surface, $E\to M$ be a bundle of rank $2$ equipped with a Riemannian metric and a compatible connection, with given spin structures. We consider $$\Sigma:=\Sigma M\otimes\Sigma E,$$ the tensor product of spinor bundles constructed from $TM$ and $E.$ We endow $\Sigma$ with the spinorial connection $$\nabla:=\nabla^M\otimes\nabla^E,$$ the tensor product of the spinor connections on $\Sigma M$ and on $\Sigma E,$ and with the natural acction of the Clifford bundle $$Cl(TM\oplus E)\simeq Cl(TM)\widehat{\otimes} Cl(E),$$ see \cite{bayard1,bayard_lawn_roth,bayard_patty}. This permits to define the Dirac operator $D$ on $\Gamma(\Sigma)$ by 
\begin{equation*}
D\varphi=-e_1\cdot\nabla_{e_1}\varphi+e_2\cdot\nabla_{e_2}\varphi,
\end{equation*} where $(e_1,e_2)$ is an orthonormal frame of $TM.$

If we denote by $Q_1$ and $Q_2$ the $SO(1,1)$ and $SO(2)$ principal bundles of the oriented and orthonormal frames of $TM$ and $E,$ and by $\tilde{Q}_1 \to Q_1$ and $\tilde{Q}_2\to Q_2$ the given spin structures on $TM$ and $E,$ then $\Sigma$ is the vector bundle associated to the $Spin(1,1)\times Spin(2)$ principal bundle $\tilde{Q}:=\tilde{Q}_1\times_M\tilde{Q}_2,$ and to the representation $\rho_1\otimes\rho_2\simeq \rho$ of the structure group $Spin(1,1)\times Spin(2),$ that is $$\Sigma=\tilde{Q}\times \HC/\rho.$$
Since the group $S^1_{\C}=Spin(1,1).Spin(2)$ belongs to $Spin(3,1),$ which preserves the complex bilinear map $H$ defined on $\HC,$ the spinor bundle $\Sigma$ is also equipped with a complex bilinear map $H$ and with a real scalar product $\la\cdot,\cdot\ra:=\Re e\ H(\cdot,\cdot)$ of signature $(4,4).$ We note that $H$ vanishes on the bundles $\Sigma^+$ and $\Sigma^-$ since $H$ vanishes on $S^+$ and $S^-.$
We also define a $\HC$-valued scalar product on $\Sigma$ by
\begin{equation}\label{prod_escal_vector}
\laa\psi,\psi'\raa:=\overline{\xi'}\xi,
\end{equation}
where $\xi$ and $\xi'\in\HC$ are respectively the components of $\psi$ and $\psi'$ in some local section of  $\tilde{Q};$ this scalar product satisfies the following properties: \begin{equation}\label{prop_prod_escal_vector} \laa\psi,\psi'\raa=\overline{\laa\psi',\psi\raa}\hspace{0.2in}\text{and}\hspace{0.2in}\laa X\cdot\psi,\psi'\raa=-\widehat{\laa\psi,X\cdot\psi'\raa}
\end{equation} for all $\psi,\psi'\in\Sigma$ and for all $X\in TM\oplus E.$ 
Note that, by definition, $H(\psi,\psi')$ is the coefficient of $\e$ in the descomposition of $\laa\psi,\psi'\raa$ in the basis $\e,I,J,K$ of $\HC,$ and that (\ref{prop_prod_escal_vector}) yields
\begin{equation}\label{aplic_H_propiedad}
H(\psi,\psi')=H(\psi',\psi) \hspace{0.2in}\text{and}\hspace{0.2in} H(X\cdot\psi,\psi')=-\overline{H(\psi,X\cdot\psi')}.
\end{equation}
\textit{Notation.} We will use the next notation: if $\tilde{s}\in\tilde{Q}$ is a given spinorial frame, the brackets $[\cdot]$ will denote the coordinates in $\HC$ of the spinor fields in the frame $\tilde{s},$ that is, for all $\varphi\in\Sigma,$
\begin{equation*}
\varphi\simeq[\tilde{s},[\varphi]]\hspace{0.2in}\in\hspace{0.2in}\Sigma\simeq\tilde{Q}\times\HC/\rho.
\end{equation*}
We will also use the brackets to denote the coordinates in $\tilde{s}$ of the elements of the Clifford algebra $Cl(TM\oplus E):$ for $X\in Cl_0(TM\oplus E)$ and $Y\in Cl_1(TM\oplus E)$ will be respectively represented by $[X],[Y]\in\HC$ such that, in $\tilde{s},$ 
\begin{equation*}X\simeq\begin{pmatrix}[X] & 0\\ 0 & \widehat{[X]}\end{pmatrix} \hspace{0.2in}\text{and}\hspace{0.2in} Y\simeq\begin{pmatrix}0 & [Y]\\ \widehat{[Y]} & 0\end{pmatrix}.\end{equation*}
Note that
\begin{equation*}
[X\cdot\varphi]=[X][\varphi]\hspace{0.2in}\text{and}\hspace{0.2in}[Y\cdot\varphi]=[Y]\widehat{[\varphi]}
\end{equation*}
and that, in a spinorial frame $\tilde{s}\in\tilde{Q}$ such that $\pi(\tilde{s})=(e_1,e_2,e_3,e_4),$ where $\pi:\tilde{Q}\to Q_1\times_M Q_2$ is the natural projection onto the bundle of the orthonormal frames of $TM\oplus E$ adapted to the splitting, $e_1,e_2,e_3$ and $e_4\in Cl_1(TM\oplus E)$ are respectively represented by $i\e, I, J$  and $K\in \HC.$

\section{Spinor representation of timelike surfaces}\label{spinor rep}
In this section we will prove the spinor representation theorem of a timelike surface in $\R^{3,1}.$ This result is the generalization of the principal theorems of \cite{bayard1,bayard_lawn_roth,bayard_patty} and completes the spinorial description of semi-Riemannian surfaces in four dimensional semi-Riemannian Euclidean space.

\subsection{The proof of Theorem \ref{principal}}
The proof of affirmations $\textit{3-}\Rightarrow \textit{2-} \Rightarrow \textit{1-}$ are given by the spinorial Gauss formula (see Section \ref{gauss_formula}). 
As in \cite{friedrich} (and after in \cite{lawn,lawn_roth,morel,roth} and in \cite{bayard1,bayard_lawn_roth,bayard_patty}) the proof of $\textit{1-}\Rightarrow\textit{3-}$ relies on the fact that such spinor field necessarily solves a Killing type equation:\\

\begin{prop}\label{lema_princ}
If $\varphi$ is a solution of $D\varphi=\vec{H}\cdot\varphi,$ with $H(\varphi,\varphi)=1,$ then $\varphi$ satisfies
\begin{equation}\label{kill} 
\nabla_X\varphi= -\frac{1}{2}\sum_{j=1,2}\epsilon_je_j\cdot B(X,e_j) \cdot\varphi,
\end{equation}
for all $X\in \Gamma(TM),$ where $B:TM\times TM\to E$ is the bilinear and symetric map defined by \begin{equation*}\label{seg_for_fund}
\la B(X,Y),\nu\ra=2\la X\cdot\nabla_Y\varphi,\nu\cdot\varphi\ra
\end{equation*}
for all $X,Y\in\Gamma(TM)$ and all $\nu\in\Gamma(E).$ 

Moreover, the map $B$ satisfies the Gauss, Ricci and Codazzi equations and is such that $\vec{H}=\frac{1}{2}tr_{\la,\ra}B.$
\end{prop}

Note that, in the proposition we use the same notation $\la\cdot,\cdot\ra$ to denote the scalar products on $TM,$ on $E$ and on $\Sigma.$

\begin{proof}
We consider the complex structure $i:=-e_1\cdot e_2\cdot e_3\cdot e_4,$ defined on the Clifford bundle $Cl(TM\oplus E)$ by the multiplication on the left, and on the spinor bundle $\Sigma$ by the Clifford action. The map $H:\Sigma\times\Sigma\to \C$ is $\C$-bilinear with respect to this complex structure, whereas the Clifford action satisfies \[i(X\cdot\varphi)=(i X)\cdot\varphi=-X\cdot(i \varphi),\] for all $\varphi\in\Sigma$ and $X\in TM\oplus E.$
Now, we consider the following spinors
\begin{equation*}\{ \varphi,\ \ e_1\cdot e_2\cdot\varphi,\ \ e_2\cdot e_3\cdot\varphi,\ \ e_3\cdot e_1\cdot\varphi\}.\end{equation*}
Using the identities in (\ref{aplic_H_propiedad}), we can show that these spinors form an $H$-orthonormal set of $\Sigma;$ in particular, for all $X\in TM$ we have
\begin{align*} \nabla_X\varphi & = H(\nabla_X\varphi,\varphi)\varphi-H(\nabla_X\varphi,e_1\cdot e_2\cdot\varphi)e_1\cdot e_2\cdot\varphi \\ & \ \ +H(\nabla_X\varphi,e_2\cdot e_3\cdot\varphi)e_2\cdot e_3\cdot\varphi\  -H(\nabla_X\varphi,e_3\cdot e_1\cdot\varphi)e_3\cdot e_1\cdot\varphi.
\end{align*}
Using $H(\varphi,\varphi)=1,$ we get $H(\nabla_X\varphi,\varphi)=0$ for all $X\in TM;$ on the other hand, using the Dirac equation $D\varphi=\vec{H}\cdot\varphi,$ we obtain $H(\nabla_X\varphi,e_1\cdot e_2\cdot\varphi)=0,$ for all $X\in TM:$ if $X=e_1$ (the case when $X=e_2$ is analogous) we have   \begin{align*}
H(\nabla_{e_1}\varphi,e_1\cdot e_2\cdot\varphi) &=-\overline{H(e_1\cdot\nabla_{e_1}\varphi,e_2\cdot\varphi)}
= \overline{H(-e_2\cdot \nabla_{e_2}\varphi+\vec{H}\cdot\varphi,e_2\cdot\varphi)}\\
&= -H(\nabla_{e_2}\varphi,\varphi)+\overline{H(\vec{H}\cdot\varphi,e_2\cdot\varphi)}=0
\end{align*} since $H(\nabla_{e_2}\varphi,\varphi)=0$ and  $$H(\vec{H}\cdot\varphi,e_2\cdot\varphi)=-\overline{H(\varphi,\vec{H}\cdot e_2\cdot\varphi)}=-H(e_2\cdot\varphi,\vec{H}\cdot\varphi)=-H(\vec{H}\cdot\varphi,e_2\cdot\varphi).$$
Thus, we can write $\nabla_X\varphi=\eta(X)\cdot\varphi,$
where $$\eta(X):=H(\nabla_X\varphi,e_2\cdot e_3\cdot\varphi)e_2\cdot e_3-H(\nabla_X\varphi,e_3\cdot e_1)e_3\cdot e_1.$$ Using the relations $i\ e_2\cdot e_3=e_1\cdot e_4$ and $i\ e_3\cdot e_1=e_4\cdot e_2,$ we can see that $\eta(X)$ has the form
\begin{equation}\label{expresi_eta}
\eta(X)=e_1\cdot \nu_1+e_2\cdot\nu_2,
\end{equation}
for some $\nu_1,\nu_2\in E.$ Now, for each $\nu \in E$ and $j=1,2$ we have
\begin{equation*}
\la B(e_j,X),\nu\ra=2\la e_j\cdot\nabla_X\varphi,\nu\cdot\varphi\ra=-2\la\nabla_X\varphi,e_j\cdot\nu\cdot\varphi\ra= -2\la\eta(X)\cdot\varphi,e_j\cdot\nu\cdot\varphi\ra,
\end{equation*}
using the expression \eqref{expresi_eta} of $\eta(X),$ we get
\begin{equation}\label{seg_form_fun_iguald}
\la B(e_j,X),\nu\ra =-2\la e_1\cdot\nu_1\cdot\varphi,e_j\cdot\nu\cdot\varphi\ra -2\la e_2\cdot\nu_2\cdot\varphi,e_j\cdot\nu\cdot\varphi\ra.
\end{equation}
We note that for all $\nu,\nu'\in E$ we have $$\la e_1\cdot e_2\cdot\varphi,\nu\cdot\nu'\cdot\varphi\ra=0$$ (the proof is analogous as Lemma 3.1 of \cite{bayard1}). Thus, the identity (\ref{seg_form_fun_iguald}) gives \begin{align*}
\la B(e_1,X),\nu\ra &=2\la\nu_1\cdot\varphi,\nu\cdot\varphi\ra=2\la\nu_1,\nu\ra,\\
\la B(e_2,X),\nu\ra &=-2\la\nu_2\cdot\varphi,\nu\cdot\varphi\ra=-2\la\nu_2,\nu\ra,
\end{align*}
therefore $\nu_1=\frac{1}{2}B(e_1,X)$ and $\nu_2=-\frac{1}{2}B(e_2,X),$ and thus, by (\ref{expresi_eta}), we obtain $$\eta(X)=-\frac{1}{2}\sum_{j=1,2}\epsilon_je_j\cdot B(X,e_j).$$ 
Finally, the Gauss, Ricci and Codazzi equations appear to be the integrability condition of \eqref{kill}. The proof is analogous to that given in \cite[Theorem 2]{bayard1} and will therefore be omitted.
\end{proof}

Using the fundamental theorem of submanifolds (see Remark \ref{teo_fund_inmersion}) we obtain the proof of the first part of Theorem \ref{principal}; the proof of the spinor representation formula is given in the next section.

\subsection{The spinor representation formula}\label{weierstrass_repre}
With the hypothesis of Theorem \ref{principal}, assume that we have a spinor field $\varphi\in \Gamma(\Sigma)$ such that \begin{equation}\label{1}
D\varphi=\vec{H}\cdot\varphi
\end{equation} with $H(\varphi,\varphi)=1.$ 
We define the $1$-form $\xi:TM\oplus E\to \HC$ by 
\begin{equation*}\label{forma_xi}
\xi(X)=\laa X\cdot\varphi,\varphi\raa\hspace{0.1in}\in\hspace{0.1in}\HC
\end{equation*} 
where the pairing $\laa.,.\raa:\Sigma\times\Sigma\to \HC$ is defined in \eqref{prod_escal_vector}. \\

\begin{prop} The $1-$form $\xi$ satisfies the following fundamental properties:
\begin{enumerate}
\item[1-] $\xi=-\widehat{\overline{\xi}},$ thus $\xi$ takes its values in $\R^{3,1}\subset\HC,$ and 

\item[2-] $\xi:TM\to \R^{3,1}$ is closed, \textit{i.e.} $d\xi=0.$
\end{enumerate}
\end{prop}
\begin{proof} The proof of the first affirmation is a consequence of the identities in \eqref{prop_prod_escal_vector} of the scalar product $\laa\cdot,\cdot\raa$ (we recall the identification \eqref{iden_espa} of $\R^{3,1}$ as a subset $\HC$). The second property is a consequence of the Dirac equation \eqref{1}; see  \cite[Proposition 4.1]{bayard1} and \cite[Lemma 2.3]{bayard_patty} for similar properties and detailed proofs.
\end{proof}

If we moreover assume that $M$ is simply connected, since $\xi:TM \to \R^{3,1}$ is a closed $1-$form   there exists a differentiable map $F:M \to \R^{3,1}$ such that $dF=\xi,$ that is $$F=\int\xi:M\to\R^{3,1}\hspace{0.2in}\mbox{where}\hspace{0.2in}\xi(X)=\laa X\cdot\varphi,\varphi\raa,$$ for all $X\in TM.$ The next theorem is fundamental: \\

\begin{thm}\label{teorema_isometria}
\textit{1-} The map $F:M\to \R^{3,1}$ is an isometric immersion. 

\noindent \textit{2-} The map $\Phi: E \longrightarrow M\times\R^{3,1}$ given by
\begin{equation*}
X\in E_m \longmapsto (F(m),\xi(X))
\end{equation*}
is an isometry between $E$ and the normal bundle $N(F(M))$ of $F(M)$ in $\R^{3,1},$ preserving connections and second fundamental form.
\end{thm}
\begin{proof} The proof is consequence of the properties of the Clifford action and is analogous to that given in \cite[Theorem 3]{bayard1} and will therefore be omitted.
\end{proof}

As in \cite{bayard1,bayard_lawn_roth,bayard_patty}, Theorem \ref{teorema_isometria} gives a spinorial proof of the fundamental theorem of submanifolds (see Remark \ref{teo_fund_inmersion}). \\

\begin{coro}\label{tfs} We may integrate the Gauss, Ricci and Codazzi equations in two steps:
\begin{enumerate}
\item[1-] first solving 
\begin{equation*}
\nabla_X\varphi=\eta(X)\cdot\varphi
\end{equation*}
where
$$\eta(X)=-\frac{1}{2}\sum_{j=1,2}\epsilon_je_j\cdot B(X,e_j),$$
there exists a solution $\varphi$ in $\Gamma(\Sigma)$ such that $H(\varphi,\varphi)=1,$ unique up to the natural right-action of $Spin(3,1)$ on $\Gamma(\Sigma),$

\item[2-] then solving \begin{equation*} dF=\xi \hspace{0.2in}\mbox{where}\hspace{0.2in} \xi(X)=\laa X\cdot\varphi,\varphi\raa, \end{equation*} the solution is unique, up to translations of $\R^{3,1}\subset\HC.$
\end{enumerate}
\end{coro}

Note that the multiplication on the right by a constant belonging to $Spin(3,1)$ in the first step, and the addition of a constant belonging to $\R^{3,1}$ in the second step, correspond to a rigid motion in $\R^{3,1}.$ Another consequence of Theorem \ref{teorema_isometria} is the following classical formula: \\

\begin{coro}\label{laplacian-immersion} The Laplacian of the isometric immersion $F:M \to \R^{3,1}$ is given by 
\begin{equation*}
\Delta F=2\vec{H}
\end{equation*} where $\vec{H}$ is the mean curvature vector of the immersion. 
\end{coro}
\begin{proof} Using the properties in \eqref{prop_prod_escal_vector} we get
\begin{equation*}
\nabla dF(e_i,e_i)=\nabla\xi(e_i,e_i)=\laa e_i\cdot \nabla_{e_i}\varphi,\varphi\raa - \widehat{\overline{\laa e_i\cdot \nabla_{e_i}\varphi,\varphi\raa}},
\end{equation*} thus, from Dirac equation \eqref{1} we obtain
\begin{equation*}
\Delta F =-\nabla dF(e_1,e_1)+\nabla dF(e_2,e_2) 
= \laa D\varphi,\varphi\raa - \widehat{\overline{\laa D\varphi,\varphi\raa}}=2\laa \vec{H}\cdot\varphi,\varphi\raa
\end{equation*} where $\xi(\vec{H})=\laa \vec{H}\cdot\varphi,\varphi\raa$ is the mean curvature vector of the immersion.
\end{proof}

Applications of the spinor representation formula in Sections \ref{app-1},\ref{app-2},\ref{app-3} and \ref{app-4} will rely on the following simple observation: assume that $F_0:M\to \R^{3,1}$ is an isometric immersion and consider $\varphi=\pm\e_{|M}$ the restriction to $M$ of the constant spinor field $+\e$ or $-\e\in\HC$ of $\R^{3,1};$ if \begin{equation}\label{spinor formula}
F=\int\xi,\hspace{0.3in} \xi(X)=\laa X\cdot\varphi,\varphi\raa
\end{equation} is the immersion given in the theorem, then $F\simeq F_0.$ This is in fact trivial since \begin{equation*}
\xi(X)=\laa X\cdot\varphi,\varphi\raa=\overline{[\varphi]}[X]\widehat{[\varphi]}=[X],
\end{equation*} in a spinorial frame $\tilde{s}$ of $\R^{3,1}$ which is above the canonical basis (in such a frame $[\varphi]=\pm \e$). The representation formula \eqref{spinor formula}, when written in moving frames adapted to the immersion, will give nontrivial formulas. 

\section{Timelike surfaces in the De Sitter space}\label{app-1}
In this section we deduce spinor characterizations of timelike surfaces in three-dimensional Minkowski space $\R^{2,1}$ and De Sitter space: in the first case, we recover the characterization given in \cite{bayard_patty}; in the second case, we obtain a new characterization which is different to the given in \cite{lawn_roth}. 

We suppose that $E=\R e_3\oplus \R e_4$ where $e_3$ and $e_4$ are unit, orthogonal and parallel sections of $E$ and such that $(e_3,e_4)$ is positively oriented. We consider the isometric embedding of $\R^{2,1}$ and the De Sitter space in $\R^{3,1}\subset \HC$ given by 
\begin{equation}\label{de sitter}
\R^{2,1}:=(K)^{\perp}\hspace{0.2in}\mbox{and}\hspace{0.2in}\mathbb{S}^{2,1}:=\left\lbrace x\in\R^{3,1} \mid \la x,x \ra=1 \right\rbrace,
\end{equation}
where $K$ is the fourth vector of the canonical basis of $\R^{3,1}\subset\HC.$ Let $\vec{H}$ be a section of $E$ and $\varphi\in\Gamma(\Sigma)$ be a solution of 
\begin{equation}\label{dirac sec 1}
D\varphi=\vec{H}\cdot\varphi,\hspace{0.3in} H(\varphi,\varphi)=1.
\end{equation}
According to Theorem \ref{principal}, the spinor field $\varphi$ defines an isometric immersion $M\to \R^{3,1}$ (unique, up to translations), with normal bundle $E$ and mean curvature vector $\vec{H}.$ We give a characterization of the isometric immersion in $\R^{2,1}$ and $\mathbb{S}^{2,1}$ (up to translations) in terms of $\varphi:$ \\
\begin{prop}\label{inm reduc}
\textit{1-} Assume that \begin{equation}\label{inm R21}
\vec{H}=He_3\hspace{0.3in}\text{and}\hspace{0.3in} e_4\cdot\varphi=\pm i\varphi. 
\end{equation}Then the isometric immersion $M\to \R^{3,1}$ belongs to $\R^{2,1}.$

\noindent \textit{2-} Consider the function $F=\laa e_4\cdot\varphi,\varphi\raa$ and assume that \begin{equation}\label{inm S21}
\vec{H}=He_3-e_4\hspace{0.3in}\text{and}\hspace{0.3in} dF(X)=\laa X\cdot\varphi,\varphi\raa. 
\end{equation}Then the isometric immersion $M\to \R^{3,1}$ belongs to $\mathbb{S}^{2,1}.$

Reciprocally, if $M\to \R^{3,1}$ belongs to $\R^{2,1}$ (resp. to $\mathbb{S}^{2,1}$), then  \eqref{inm R21} (resp. \eqref{inm S21}) holds for some unit, orthogonal and parallel sections $(e_3,e_4)$ of $E.$
\end{prop}
\begin{proof}
\textit{1-} We suppose that \eqref{inm R21} holds, and we compute 
\begin{equation*}
\xi(e_4)=\laa e_4\cdot\varphi,\varphi\raa=\pm \laa i\varphi,\varphi\raa=\pm \overline{[\varphi]}([\varphi]K)=\pm K.
\end{equation*}
The constant vector $K$ is thus normal to the immersion (by Theorem \ref{teorema_isometria}, since this is $\xi(e_4)$), and the result follows. \\
\textit{2-} Analogously, assuming that \eqref{inm S21} holds, the function $F=\laa e_4\cdot\varphi,\varphi\raa$ is a primitive of the $1-$form $\xi(X)=\laa X\cdot\varphi,\varphi\raa,$ and is thus the isometric immersion defined by $\varphi$ (uniquely defined, up to translations); since the Minkowski norm of $\laa e_4\cdot\varphi,\varphi\raa\in\R^{3,1}\subset \HC$ coincides with the norm of $e_4,$ and is thus constant equal to $1,$ the immersion belongs to $\mathbb{S}^{2,1}.$

For the converse statements, we choose $(e_3,e_4)$ such that $\laa e_4\cdot\varphi,\varphi\raa=\pm K$ in the first case and such that $\laa e_4\cdot\varphi,\varphi\raa$ is the normal vector to $\mathbb{S}^{2,1}$ in $\R^{3,1}$ in the second case. Writing these identities in some frame $\tilde{s},$ we easily deduce \eqref{inm R21} and \eqref{inm S21}.
\end{proof}
We now assume that $M\subset \mathcal{H}\subset\R^{3,1},$ where $\mathcal{H}$ is $\R^{2,1}$ or $\mathbb{S}^{2,1},$ and consider $e_3$ and $e_4$ unit vector fields such that 
\begin{equation*}
\R^{3,1}=T\mathcal{H}\oplus_{\perp}\R e_4 \hspace{0.2in}\mbox{and}\hspace{0.2in} T\mathcal{H}=TM\oplus_{\perp} \R e_3.
\end{equation*} 
The intrinsic spinors of $M$ indentify with the spinors of $\mathcal{H}$ restricted to $M,$ which in turn identify with the positive spinors of $\R^{3,1}$ restricted to $M:$ this is the content of Proposition \ref{ident spinors} below, which, together with the previous result, will give the representation of timelike surfaces in $\R^{2,1}$ and $\mathbb{S}^{2,1}$ by means of spinors of $\Sigma M$ only.

We define the scalar product on $\C^2$ by setting
\begin{equation*}
\left\la \begin{pmatrix}
a+ib \\ c+id
\end{pmatrix},\begin{pmatrix}
a'+ib' \\ c'+id'
\end{pmatrix} \right\ra:=\frac{ad'+a'd-bc'-b'c}{2},
\end{equation*} of signature $(2,2).$ This scalar product is $Spin(1,1)$-invariant, thus induces a scalar product $\la\cdot,\cdot\ra$ on the spinor bundle $\Sigma M.$ It satisfies the following properties:
\begin{equation*}\label{prop pp}
\la\psi,\psi'\ra=\la\psi',\psi\ra\hspace{0.2in}\text{and}\hspace{0.2in}\la X\cdot_M\psi,\psi'\ra=-\la\psi,X\cdot_M\psi'\ra, 
\end{equation*} for all $\psi,\psi'\in \Sigma M$ and all $X\in TM.$ This is the scalar product on $\Sigma M$ that we use in this section (and in this section only). We moreover define $|\psi|^2:=\la\psi,\psi\ra$ and, we denote by $i$ the natural complex structure of $\Sigma M,$ which is such that the Clifford action is $\C-$linear. The following proposition is analogous to the given in \cite{bayard1,bayard_lawn_roth,bayard_patty} (see also \cite[Proposition 2.1]{morel}, and  the references therein). \\

\begin{prop}\label{ident spinors}
There exists an identification 
\begin{align*}
\Sigma M & \overset{\sim}{\longmapsto}  \Sigma^+_{|M}\\
\psi & \longmapsto  \psi^*
\end{align*}
such that, for all $X\in TM$ and all $\psi\in\Sigma M,$ $(\nabla_X\psi)^*=\nabla_X\psi^*,$ the Clifford actions are linked by 
\begin{equation*}\label{rela clifford}
(X\cdot_M\psi)^*=X\cdot e_3\cdot\psi^* \end{equation*} and the following two properties holds: 
\begin{equation}\label{norma R21}
H(\psi^*,i e_4\cdot \psi^*)=-\frac{1}{2} |\psi|^2, 
\end{equation} and \begin{equation}\label{norm deri}
d\laa e_4\cdot \psi^*,\psi^*\raa(X)=\laa X\cdot\psi^*,\psi^*\raa \hspace{0.2in}\mbox{iff}\hspace{0.2in} 
d \left( |\psi|^2\right)(X)=\la i(X\cdot_M\overline{\psi}),\psi\ra.
\end{equation}
\end{prop}

Using this identification, the intrinsic Dirac operator on $M,$ defined by \begin{equation*}D_M\psi:=-e_1\cdot_M\nabla_{e_1}\psi+e_2\cdot_M\nabla_{e_2}\psi, \end{equation*} where $(e_1,e_2)$ is an orthogonal basis tangent to $M$ such that $|e_1|^2=-1$ and $|e_2|^2=1,$ is linked to $D$ by 
\begin{equation}\label{relac dirac0}
(D_M \psi)^*=-e_3\cdot D\psi^*
\end{equation}
We suppose that $\varphi\in\Gamma(\Sigma)$ is a solution of equation \eqref{dirac sec 1}, we may consider $\psi\in \Sigma M$ such that $\psi^*=\varphi^+;$ it satisfies 
\begin{equation}\label{relac dirac}
(D_M \psi)^*=-e_3\cdot D\psi^*=-e_3\cdot\vec{H}\cdot\psi^*.
\end{equation}
Note that $\psi\neq 0,$ since 
\begin{equation}\label{norma rela}
H(\varphi,\varphi)=2H(\varphi^+,\varphi^-)=1,
\end{equation} where the descomposition $\varphi=\varphi^++\varphi^-$ is the descomposition in $\Sigma=\Sigma^+\oplus\Sigma^-,$ and recall that $H$ vanishes on $\Sigma^+$ and $\Sigma^-;$ see Section \ref{twisted sec}. 

We consider the case of a timelike surface in $\R^{2,1},$ \textit{i.e.} $\mathcal{H}=\R^{2,1}.$ Then, $\vec{H}$ is of the form $He_3$ and \eqref{relac dirac} reads \begin{equation*}
D_M\psi=H\psi;
\end{equation*} moreover, \eqref{norma rela}, \eqref{inm R21} and \eqref{norma R21} imply that $|\psi|^2=\pm 1.$ This is the spinorial characterization of an isometric immersion in $\R^{2,1}$ given in \cite{bayard_patty}.

Now, we examine the case of a timelike surface in $\mathbb{S}^{2,1}.$  If $\mathcal{H}=\mathbb{S}^{2,1},$ then $\vec{H}$ is of the form $He_3-e_4,$ and using \eqref{relac dirac} we get 
\begin{equation*}
(D_M \psi)^*=-e_3\cdot\vec{H}\cdot\psi^*=-e_3\cdot (He_3-e_4)\cdot\psi^*=H\psi^*+e_3\cdot e_4\cdot \psi^*=H\psi^*-(i\overline{\psi})^*,
\end{equation*} where $\overline{\psi}=\psi^+-\psi^-$ denotes the usual conjugation in $\Sigma M.$ 
Moreover, it is not difficult to prove that \eqref{inm S21} implies that \eqref{norm deri} holds. We thus get 
\begin{equation}\label{rep S21}
D_M\psi=H\psi-i\overline{\psi}\hspace{0.2in}\mbox{and}\hspace{0.2in} d \left( |\psi|^2\right)(X)=\la i(X\cdot_M\overline{\psi}),\psi\ra.
\end{equation}
Reciprocally, let $M$ be a timelike surface and $H:M\to\R$ a given differentiable function, and suppose that $\psi\in\Gamma(\Sigma M)$ satisfies \eqref{rep S21}. We define $\varphi^+:=\psi^*\in\Sigma^+$ and $\vec{H}:=He_3-e_4,$ where $e_3$ and $e_4$ are unit, orthogonal and parallel sections of $E,$ and such that $(e_3,e_4)$ is positively oriented. Using \eqref{rep S21}, \eqref{relac dirac0} and \eqref{norm deri} we obtain
\begin{equation*}
D\varphi^+=\vec{H}\cdot\varphi^+ \hspace{0.2in}\mbox{and}\hspace{0.2in} d\laa e_4\cdot \psi^*,\psi^*\raa(X)=\laa X\cdot\psi^*,\psi^*\raa.
\end{equation*}
\begin{prop} Let $\psi\in\Gamma(\Sigma M)$ be a solution of \eqref{rep S21}. There exists a spinor field $\varphi\in\Gamma(\Sigma)$ solution of 
\begin{equation*}
D\varphi=\vec{H}\cdot\varphi \hspace{0.2in}\mbox{and}\hspace{0.2in} H(\varphi,\varphi)=1,
\end{equation*} with $\varphi^+=\psi^*$ and such that the immersion defined by $\varphi$ is given by $F=\laa e_4\cdot\varphi,\varphi\raa.$ In particular $F(M)$ belongs to $\mathbb{S}^{2,1}.$
\end{prop}
\begin{proof} We need to find $\varphi^-$ solution of the system 
\begin{eqnarray*}
F_1 =\laa e_4\cdot\varphi^-,\varphi^-\raa \\
dF_1(X) =\laa X\cdot\varphi^-,\varphi^- \raa
\end{eqnarray*} with $[\varphi^+]\overline{[\varphi^-]}=\frac{1}{2};$ this system is equivalent to \begin{equation*}
\varphi^-=-2e_4\cdot (\varphi^+\bullet F_1),
\end{equation*} with $2[\varphi^+\bullet F_1]\widehat{\overline{[\varphi^+]}}=K,$ where $F_1:M\to \HC$ solves the equation \begin{equation}\label{ecuac exist}
\varphi^+\bullet dF_1(X)=-\omega(X)\cdot (\varphi^+\bullet F_1), 
\end{equation} where $\omega(X)=X\cdot e_4,$ for all $X\in TM.$ Above, $\bullet$ means the natural action of $\HC$ on $\Sigma$ on the right given in coordinates by $[\varphi \bullet q]=[\varphi]q.$ The compatibility equation of \eqref{ecuac exist} is given by
\begin{equation}\label{comp}
[\omega(X),\omega(Y)]=[\omega(X),\eta(Y)]-[\omega(Y),\eta(X)],
\end{equation} 
where $\eta$ is such that $\nabla_X\varphi^+=\eta(X)\cdot\varphi^+,$ and where $[p,p']=pp'-p'p,$ for all $p,p'\in Cl_0(3,1);$ by a direct computation \eqref{comp} is satisfied, and thus \eqref{ecuac exist} is solvable.
\end{proof}

\begin{remar}\label{S21}
A solution of \eqref{rep S21} is thus equivalent to an isometric immersion in three-dimensional De Sitter space $\mathbb{S}^{2,1}.$ We thus obtain a spinorial characterization of an isometric immersion of a timelike surface in $\mathbb{S}^{2,1},$ which is simpler than the characterization given in \cite{lawn_roth}, where two spinor fields are needed. 
\end{remar}

\section{The Laplacian of the Gauss map of a timelike surface in $\R^{3,1}$}\label{app-2}

The main goal of this section is to compute the Laplacian of the Gauss map of a timelike surface in $\R^{3,1}.$ 

\subsection{The Grassmannian of the timelike planes in $\R^{3,1}$}\label{gauss section}
The Grassmannian of the oriented timelike planes in $\R^{3,1}$ identifies to 
\begin{equation*}
\Q=\left\lbrace u_1\cdot u_2 \mid u_1,u_2\in\R^{3,1}, |u_1|^2=-|u_2|^2=-1 \right\rbrace \  \subset \ Cl_0(3,1). 
\end{equation*}  Setting \begin{equation*}
\Im m\HC:= \C iI\ \oplus\ \C J\ \oplus\ \C iK
\end{equation*} and since $e_1\cdot e_2\simeq iI, e_2\cdot e_4\simeq -J$ and $e_4\cdot e_1\simeq -iK$ in the identification $Cl_0(3,1)\simeq \HC$ given in \eqref{elempar}, we easily get 
\begin{equation*}
\Q=\left\lbrace \xi\in \Im m \HC \mid H(\xi,\xi)=-1 \right\rbrace.
\end{equation*} We define the \textit{cross product} of two vectors $\xi,\xi'\in \Im m\HC$ by \begin{equation*}
\xi\times\xi':=\frac{1}{2}(\xi\xi'-\xi'\xi) \in \Im m\HC.
\end{equation*} 
We also define the \textit{mixed product} of three vectors $\xi,\xi',\xi''\in\Im m\HC$ by \begin{equation*}
[\xi,\xi',\xi'']:=H(\xi\times\xi',\xi'') \in \C;
\end{equation*} it is easily seen to be, up to sign, the determinant of the vectors $\xi,\xi',\xi''\in\Im m\HC$ in the basis $(iI,J,iK)$ of $\Im m\HC$ (considered as a complex space). The mixed product is a complex volume form on $\Im m\HC,$ and induces a natural \textit{complex area form} $\omega_{\Q}$ on $\Q$ by \begin{equation*}
\omega_{\Q}(p)(\xi,\xi'):=[\xi,\xi',p],
\end{equation*} for all $p\in\Q$ and all $\xi,\xi'\in T_p\Q.$ Note that $\omega_{\Q}(p)(\xi,\xi')=0$ if and only if $\xi$ and $\xi'$ are dependent over $\C.$

\subsection{The Gauss map of a timelike surface in $\R^{3,1}$}

Let $M$ be an oriented timelike surface in $\R^{3,1}.$ We consider its Gauss map
\begin{eqnarray*}
G: M & \longrightarrow & \Q \\
x & \longmapsto & u_1\cdot u_2
\end{eqnarray*} where, at $x\in M,$ $(u_1,u_2)$ is a positively oriented orthogonal basis of $T_xM$ such that $|u_1|^2=-|u_2|^2=-1.$ The pull-back by the Gauss map of the area form $\omega_{\Q}$ is given by the following proposition (see similar results in \cite{bayard1,bayard_patty}):\\
\begin{prop}\label{gauss0} We have \begin{equation*}
G^*\omega_{\Q}=(K+iK_N)\ \omega_M,
\end{equation*} where $\omega_M$ is the area form, $K$ is the Gauss curvature and $K_N$ is the normal curvature of $M.$ In particular, assuming moreover that 
\begin{equation*}\label{der gauss}
dG_x:T_xM \longrightarrow T_{G(x)}\Q
\end{equation*} is one-to-one at some point $x\in M,$ then $K=K_N=0$ at $x$ if and only if the linear space $dG_x(T_xM)$ is a complex line in $T_{G(x)}\Q,$ \textit{i.e.} 
\begin{equation}\label{linea}
dG_x(T_xM)=\{ z\ U \mid z\in \C \}
\end{equation} where $U$ is some vector belonging to $T_{G(x)}\Q\subset \HC.$
\end{prop}

As a consequence of this proposition, if $K=K_N=0$ and $G:M\to\Q$ is a regular map (\textit{i.e.} if $dG_x$ is injective at every point $x$ of $M$), there is a unique complex structure $\mathcal{J}$ on $M$ such that 
\begin{equation}\label{complex}
dG_x(\mathcal{J} u)=i\ dG_x(u)
\end{equation}
for all $x\in M$ and all $u\in T_xM.$ Indeed, \eqref{linea} implies that $dG_x(T_xM)$ is stable by multiplication by $i,$ and we may define 
\begin{equation*}
\mathcal{J} u:=dG_x^{-1}(i\ dG_x(u)).
\end{equation*}
This complex structure coincides with the complex structure considered in \cite{AGM}, and we will use this in Section \ref{app-3}.

\subsection{The Laplacian of the Gauss map}

We suppose that the immersion of $M$ in $\R^{3,1}$ is given by some spinor field $\varphi \in\Gamma(\Sigma)$ solution of the Dirac equation 
$D\varphi=\vec{H}\cdot\varphi$ with $H(\varphi,\varphi)=1.$ We first express the Gauss map of the immersion in terms of $\varphi:$ \\
\begin{lem}\label{gauss-varphi} The Gauss map is given by 
\begin{eqnarray*}
G : M & \longrightarrow & \Q \\
x & \longmapsto & \laa e_1\cdot e_2 \cdot \varphi,\varphi \raa
\end{eqnarray*} where, for all $x\in M,$  $(e_1,e_2)$ is a positively oriented and orthonormal basis of $T_xM.$
\end{lem}
\begin{proof} Setting $u_1=\laa e_1\cdot\varphi,\varphi\raa$ and $u_2=\laa e_2\cdot\varphi,\varphi\raa\in\R^{3,1}\subset \HC,$ the basis $(u_1,u_2)$ is an orthonormal basis of the immersion  (Theorem \ref{teorema_isometria}), and 
\begin{align*}
u_1\cdot u_2 \simeq u_1 \widehat{u}_2 &=\laa e_1\cdot\varphi,\varphi\raa \widehat{\laa e_2\cdot\varphi,\varphi\raa}=\left( \overline{[\varphi]}[e_1]\widehat{[\varphi]}\right) \left( \widehat{\overline{[\varphi]}[e_1]\widehat{[\varphi]}} \right) \\
&= \overline{[\varphi]}[e_1]\widehat{[e_2]}[\varphi]=\laa e_1\cdot e_2\cdot\varphi,\varphi\raa,
\end{align*} where $[e_1], [e_2]$ and $[\varphi]\in\HC$ represent $e_1,e_2$ and $\varphi$ in some frame $\tilde{s}$ of $\tilde{Q}.$
\end{proof}
According to Theorem \ref{principal}, the spinor field $\varphi$ also satisfies $\nabla_X\varphi=\eta(X)\cdot \varphi,$ for all $X\in TM,$ where 
\begin{equation}\label{eta-sec-gauss}
\eta(X)=-\frac{1}{2}\sum_{j=1,2}\epsilon_j e_j \cdot B(X,e_j);
\end{equation} the second fundamental form $B$ was defined in Proposition \ref{lema_princ}. The differential of the Gauss map is linked to the second fundamental form $B$ as follows: \\
\begin{lem}\label{derivada-gauss-varphi} The $1-$form $\tilde{\eta}:=\laa \eta\cdot\varphi,\varphi\raa$ satisfies 
\begin{equation*}
dG=2G\tilde{\eta}.
\end{equation*}
\end{lem} 
\begin{proof} We suppose that $(e_1,e_2)$ is a moving frame on $M$ such that $\nabla e_{i|p}=0$ and compute \begin{align*}
dG(X)&=\laa e_1\cdot e_2\cdot \nabla_X\varphi,\varphi\raa + \laa e_1\cdot e_2\cdot \varphi,\nabla_X\varphi\raa \\ 
&= \laa e_1\cdot e_2\cdot \eta(X)\cdot \varphi,\varphi\raa + \laa e_1\cdot e_2\cdot \varphi,\eta(X)\cdot\varphi\raa \\
&= 2\laa e_1\cdot e_2\cdot \eta(X)\cdot \varphi,\varphi \raa.
\end{align*} But 
\begin{align*}
\laa e_1\cdot e_2\cdot \eta(X)\cdot \varphi,\varphi \raa &= \overline{[\varphi]} [e_1\cdot e_2] [\eta(X)] [\varphi] 
= \laa e_1\cdot e_2\cdot \varphi,\varphi \raa \laa \eta(X)\cdot \varphi,\varphi \raa
\end{align*} where $[\varphi], [e_1\cdot e_2]$ and $[\eta(X)]\in \HC$ represent $\varphi, e_1\cdot e_2$ and $\eta(X)$ respectively in some local frame $\tilde{s}$ of $\tilde{Q}.$
\end{proof}

Using the lemma above, in the same moving frame, the Laplacian of the Gauss map $G$ seen as a map from $M$ to $\HC$ (\textit{i.e.} we will use the connection of the ambient space instead of the connection of the Grassmannian) is given by 
\begin{align}\label{laplacian-gauss}
\Delta G&=-\nabla dG(e_1,e_1)+\nabla dG(e_2,e_2) \notag \\ 
&= -2 \left[ e_1(G\tilde{\eta}(e_1))-G\tilde{\eta}(\nabla_{e_1}e_1) \right]+2 \left[ e_2(G\tilde{\eta}(e_2))-G\tilde{\eta}(\nabla_{e_2}e_2) \right] \notag \\
&=2G \left( -2\tilde{\eta}(e_1)\tilde{\eta}(e_1)+2\tilde{\eta}(e_2)\tilde{\eta}(e_2) - e_1(\tilde{\eta}(e_1)) + e_2(\tilde{\eta}(e_2))\right).  
\end{align}
Now, we note that 
\begin{align*}
-\tilde{\eta}(e_1)\tilde{\eta}(e_1)+\tilde{\eta}(e_2)\tilde{\eta}(e_2) 
&= -\overline{[\varphi]}[\eta(e_1)]\overline{[\varphi]}[\varphi][\eta(e_1)][\varphi] + 
\overline{[\varphi]}[\eta(e_2)]\overline{[\varphi]}[\varphi][\eta(e_2)][\varphi] \\
&= -\overline{[\varphi]}[\eta(e_1)\cdot \eta(e_1)][\varphi] + 
\overline{[\varphi]}[\eta(e_2)\cdot \eta(e_2)][\varphi] \\
&= \overline{[\varphi]}[-\eta(e_1)\cdot \eta(e_1)+\eta(e_2)\cdot \eta(e_2)][\varphi] \\
&= \laa \left(-\eta(e_1)\cdot \eta(e_1)+\eta(e_2)\cdot \eta(e_2)\right)\cdot\varphi,\varphi \raa;
\end{align*} and 
\begin{align*}
- e_1(\tilde{\eta}(e_1)) + e_2(\tilde{\eta}(e_2)) 
&= -\laa \nabla_{e_1}\eta(e_1)\cdot\varphi,\varphi\raa - \laa\eta(e_1)\cdot\varphi,\nabla_{e_1}\varphi\raa
\\ & \ \ \ +\laa \nabla_{e_2}\eta(e_2)\cdot\varphi,\varphi\raa + \laa\eta(e_2)\cdot\varphi,\nabla_{e_2}\varphi\raa \\ &= \laa \left( -\nabla_{e_1}\eta(e_1)+\nabla_{e_2}\eta(e_2)\right)\cdot\varphi,\varphi \raa \\ & \ \ \ - \laa\eta(e_1)\cdot\varphi,\eta(e_1)\cdot\varphi\raa + \laa\eta(e_2)\cdot\varphi,\eta(e_2)\cdot\varphi\raa, 
\end{align*} but
\begin{align*}
\laa\eta(e_i)\cdot\varphi,\eta(e_i)\cdot\varphi\raa &=  \overline{[\eta(e_i)][\varphi]}[\eta(e_i)][\varphi] = \overline{[\varphi]}\ \overline{[\eta(e_i)]}[\eta(e_i)][\varphi] \\
&= - \overline{[\varphi]}[\eta(e_i)][\eta(e_i)][\varphi]= -\overline{[\varphi]}[\eta(e_i)\cdot\eta(e_i)][\varphi] \\
&= -\laa \eta(e_i)\cdot\eta(e_i)\cdot\varphi,\varphi\raa, 
\end{align*}
thus 
\begin{align*}
- e_1(\tilde{\eta}(e_1)) + e_2(\tilde{\eta}(e_2))& = \laa \left( -\nabla_{e_1}\eta(e_1)+\nabla_{e_2}\eta(e_2)\right)\cdot\varphi,\varphi \raa \\ & \ \ \ + \laa \left( \eta(e_1)\cdot\eta(e_1)-\eta(e_2)\cdot\eta(e_2)\right)\cdot\varphi,\varphi\raa.
\end{align*} Replacing this equalities in \eqref{laplacian-gauss}, we get 
\begin{align}\label{laplacian-gauss1}
\Delta G &= 2G \laa \left(-\eta(e_1)\cdot \eta(e_1)+\eta(e_2)\cdot \eta(e_2)\right) \cdot\varphi,\varphi \raa \notag \\ & \ \ \ + 2G\laa \left( -\nabla_{e_1}\eta(e_1)+\nabla_{e_2}\eta(e_2) \right) \cdot\varphi,\varphi \raa .
\end{align} 
We finally compute this terms using the following lemma. \\
\begin{lem} We have the following identities: 
\begin{enumerate}
\item[\it 1-] $-\eta(e_1)\cdot \eta(e_1)+\eta(e_2)\cdot \eta(e_2)=-|\vec{H}|^2+\frac{K}{2}-\frac{K_N}{2} e_1\cdot e_2\cdot e_3\cdot e_4$ 

\item[\it 2-] $-\nabla_{e_1}\eta(e_1)+\nabla_{e_2}\eta(e_2)=e_1\cdot \nabla_{e_1}\vec{H}-e_2\cdot \nabla_{e_2}\vec{H}$
\end{enumerate}
\end{lem}
\begin{proof} Using the expression of $\eta$ given in \eqref{eta-sec-gauss}, we have 
\begin{equation*}
\eta(e_i)=\frac{1}{2}(e_1\cdot B_{i1}-e_2\cdot B_{i2}) \hspace{0.2in}\mbox{where}\hspace{0.2in} B_{ij}=B(e_i,e_j). 
\end{equation*}
{\it 1-} By a direct computation we get 
\begin{equation*}
-\eta(e_1)\eta(e_1)=\frac{1}{4}\left( B_{11}^2-B_{12}^2+e_1\cdot e_2\cdot (B_{12}\cdot B_{11}-B_{11}\cdot B_{12}) \right)
\end{equation*} and 
\begin{equation*}
\eta(e_2)\eta(e_2)=\frac{1}{4}\left( B_{22}^2-B_{12}^2+e_1\cdot e_2\cdot (B_{12}\cdot B_{22}-B_{22}\cdot B_{12}) \right).
\end{equation*} Using the Guass and Ricci equations, we easily get 
\begin{equation*}
B_{11}^2-2B_{12}^2+B_{22}^2=-4|\vec{H}|^2+2K;
\end{equation*} and 
\begin{equation*}
B_{12}\cdot B_{11}-B_{11}\cdot B_{12}+B_{12}\cdot B_{22}-B_{22}\cdot B_{12}=-2K_N e_3\cdot e_4.
\end{equation*} 
{\it 2-} We have
\begin{equation*}
-\nabla_{e_1}\eta(e_1)+\nabla_{e_2}\eta(e_2)=\frac{1}{2} \left( -e_1\cdot \nabla_{e_1}B_{11}+e_2\cdot \nabla_{e_1}B_{12}+e_1\cdot \nabla_{e_2}B_{12}-e_2\cdot\nabla_{e_2}B_{22} \right);
\end{equation*} using the Codazzi equation (recall that $(e_1,e_2)$ is a moving frame on $TM$ such that $\nabla e_{i|p}=0$) we obtain 
\begin{equation*}
\nabla_{e_1}B_{12}=\nabla_{e_2}B_{11} \hspace{0.2in}\mbox{and}\hspace{0.2in} \nabla_{e_2}B_{12}=\nabla_{e_1}B_{22},
\end{equation*} and since $\vec{H}=\frac{1}{2}(-B_{11}+B_{22})$ we obtain the result.
\end{proof}

Therefore, using the identities of the lemma above, we get 
\begin{align*}
\laa \left(-\eta(e_1)\cdot \eta(e_1)+\eta(e_2)\cdot \eta(e_2)\right) \cdot\varphi,\varphi \raa & = \left( -|\vec{H}|^2+\frac{K}{2}\right) \laa \varphi,\varphi\raa \\ &\ \ \ + \frac{K_N}{2}\laa -e_1\cdot e_2\cdot e_3\cdot e_4\cdot\varphi,\varphi\raa 
\end{align*} since $\laa\varphi,\varphi\raa=H(\varphi,\varphi)=1,$ and since the scalar product $\laa\cdot,\cdot\raa$ is $\C-$bilinear if $\Sigma$ is endowed with the complex structure given by the Clifford action of $-e_1\cdot e_2\cdot e_3\cdot e_4$ (which corresponds to the multiplication by $i$ on $\HC$), we finally get  
\begin{equation*}
\laa \left(-\eta(e_1)\cdot \eta(e_1)+\eta(e_2)\cdot \eta(e_2)\right) \cdot\varphi,\varphi \raa= \left( -|\vec{H}|^2+\frac{K}{2}\right)+i\frac{K_N}{2}.
\end{equation*}
 
On the other hand, using the Dirac equation $D\varphi=\vec{H}\cdot\varphi,$ we have 
\begin{align}\label{parallel}
D(\vec{H}\cdot\varphi) &=-e_1\cdot \nabla_{e_1}(\vec{H}\cdot\varphi)+e_2\cdot \nabla_{e_2}(\vec{H}\cdot\varphi) \notag \\
&= -\left( e_1\cdot \nabla_{e_1}\vec{H}-e_2\cdot \nabla_{e_2}\vec{H} \right)\cdot\varphi-\vec{H}\cdot D\varphi \notag \\
&= -\left( e_1\cdot \nabla_{e_1}\vec{H}-e_2\cdot \nabla_{e_2}\vec{H} \right)\cdot\varphi+|\vec{H}|^2\varphi,
\end{align} thus 
\begin{align*}
\laa \left( -\nabla_{e_1}\eta(e_1)+\nabla_{e_2}\eta(e_2) \right) \cdot\varphi,\varphi \raa &=|\vec{H}|^2\laa \varphi,\varphi\raa-\laa D(\vec{H}\cdot\varphi),\varphi\raa \\
&= |\vec{H}|^2-\laa D(\vec{H}\cdot\varphi),\varphi\raa.
\end{align*}

Finally, replacing this expressions in \eqref{laplacian-gauss1} we obtain the formula for the Laplacian of the Gauss map \begin{equation*}
\Delta G=G(K+iK_N)-2G\laa D(\vec{H}\cdot\varphi),\varphi\raa.
\end{equation*}

As a consequence of this, if the immersion $M\subset\R^{3,1}$ have parallel mean curvature vector, using \eqref{parallel} we get
\begin{equation}\label{g}
\Delta G=(-2|\vec{H}|^2+K+iK_N)G.
\end{equation}

This formula generalizes a classical result for surfaces in Euclidean space with constant mean curvature whose Gauss map is seen as a map from the surface in $\R^3;$ see \cite{BB}. As a particular case of \eqref{g}, we obtain the following result concerning the Laplacian of the Gauss map of a minimal timelike surface in $\R^{3,1}.$  \\

\begin{coro}\label{laplacian-gauss-minima} Assume that $M$ is a minimal timelike surface in $\R^{3,1}.$ Then the Laplacian of its Gauss map is given by the following formula 
\begin{equation*}
\Delta G = (K+iK_N) G
\end{equation*} where $K$ and $K_N$ are the Gauss and normal curvatures of the surface.
\end{coro}

\section{Flat timelike surfaces with flat normal bundle and regular Gauss map in $\R^{3,1}$}\label{app-3}
We suppose that $M$ is simply connected and that the bundles $TM$ and $E$ are flat ($K=K_N=0$). Recall that the spinor bundle $\Sigma:=\Sigma M\otimes\Sigma E$ is associated to the principal bundle $\tilde{Q}$ and to the representation $\rho$ of the structure group $Spin(1,1)\times Spin(2)$ in $\HC$ (Section \ref{twisted sec}). Since the curvatures $K$ and $K_N$ are zero, the spinorial connection on the bundle $\tilde{Q}$ is flat, and $\tilde{Q}$ admits a parallel local section $\tilde{s};$ since $M$ is simply connected, the section $\tilde{s}$ is in fact globally defined. We consider $\varphi\in\Gamma(\Sigma)$ a solution of the Dirac equation 
\begin{equation*}\label{dirac sec 2}
D\varphi=\vec{H}\cdot\varphi
\end{equation*} such that $H(\varphi,\varphi)=1,$ and define $g:=[\varphi]:M \to Spin(3,1)\subset\HC$ such that \begin{equation*}
\varphi=[\tilde{s},g]\hspace{0.1in} \in\ \Sigma=\tilde{Q}\times \HC /\rho,
\end{equation*} that is, $g$ in $\HC$ represents $\varphi$ in the parallel section $\tilde{s}.$ Recall that, by Theorem \ref{principal}, $\varphi$ also satisfies 
\begin{equation}\label{killing}
\nabla_X\varphi=\eta(X)\cdot\varphi
\end{equation} for all $X\in TM,$ where 
\begin{equation}\label{sff}
\eta(X)=-\frac{1}{2}\sum_{j=1,2} \epsilon_j e_j\cdot B(X,e_j)
\end{equation} for some bilinear map $B:TM\times TM \to E.$ 

In the following, we will denote by $(e_1,e_2)$ and $(e_3,e_4)$ the parrallel, orthonormal and positively oriented frames, respectively tangent, and normal to $M,$ corresponding to $\tilde{s},$ \textit{i.e.} such that $\pi(\tilde{s})=(e_1,e_2,e_3,e_4)$ where $\pi:\tilde{Q}\to Q_1\times Q_2$ is the natural projection. We moreover assume that the Gauss map $G$ of the immersion defined by $\varphi$ is regular, and consider the complex structure $\mathcal{J}$ induced on $M$ by $G,$ defined by \eqref{complex}.

Below we will prove that $g:M\to Spin(3,1)\subset\HC$ is a holomorphic map and that the immersion defined by $\varphi$ depends on two holomorphic maps and two smooth functions. We need the following lemmas; see in \cite{bayard1,bayard_patty} similar results. \\

\begin{lem}\label{gauss map1} The Gauss map of the immersion defined by $\varphi$ is given by 
\begin{eqnarray}\label{gauss phi}
G: M & \longrightarrow & \Q \  \subset \ \Im m\HC \\
x & \longmapsto & i g^{-1}I g \notag
\end{eqnarray} where $g:[\varphi]:M \to Spin(3,1)\subset \HC$ represents $\varphi$ in some local section of $\tilde{Q}.$
\end{lem}
\begin{proof}This is the identity given in Lemma \ref{gauss-varphi} \begin{equation*}
G=\laa e_1\cdot e_2\cdot\varphi,\varphi \raa
\end{equation*} written in a section of $\tilde{Q}$ above $(e_1,e_2).$
\end{proof}

\begin{lem} Denoting by $[\eta]\in\Omega^1(M,\HC),$ the $1-$form which represents $\eta$ in $\tilde{s},$ we have 
\begin{equation}\label{eta0}
[\eta]=dg\ g^{-1}=\theta_1 J+\theta_2\ iK,
\end{equation} where $\theta_1$ and $\theta_2$ are two complex $1-$forms on $M.$
\end{lem}
\begin{proof} This is \eqref{killing} in the parallel frame $\tilde{s},$ taking into account the special form \eqref{sff} of $\eta$ for the last equality.
\end{proof}

\begin{lem} The $1-$form $\tilde{\eta}:= \laa \eta\cdot\varphi,\varphi\raa$
satisfies 
\begin{equation*}
\tilde{\eta}=\frac{1}{2}G^{-1}dG= g^{-1}dg.
\end{equation*}
\end{lem}
\begin{proof} Writing $\tilde{\eta}$ in $\tilde{s}$ together with \eqref{eta0} imply that $\tilde{\eta}=g^{-1}dg.$ The other identity was given in Lemma \ref{derivada-gauss-varphi}.
\end{proof}

The properties \eqref{gauss phi} and \eqref{eta0} may be rewritten as follows (see similar results in \cite{bayard1,bayard_patty}):\\

\begin{lem}\label{lema p} Consider the projection 
\begin{eqnarray*}
p: Spin(3,1)\subset \HC & \longrightarrow & \Q\ \subset \Im m\HC \\
g & \longmapsto & i g^{-1} Ig
\end{eqnarray*}
as a $S^1_{\C}-$principal bundle, where the action of $S^1_{\C}$ on $Spin(3,1)$ is given by the multiplication on the left. It is equipped with the horizontal distribution given at every $g\in Spin(3,1)$ by \begin{equation*}\label{distribution}
\mathcal{H}_g:=d(R_{g^{-1}})^{-1}_g(\C J\oplus \C iK)\ \subset\ T_gSpin(3,1),
\end{equation*} where $R_{g^{-1}}$ stands for the right multiplication by $g^{-1}$ on $Spin(3,1).$ The distribution $(\mathcal{H}_g)_{g\in Spin(3,1)}$ is $H-$orthogonal to the fibers of $p$, and, for all $g\in Spin(3,1),$ $dp_g:\mathcal{H}_g\to T_{p(g)}\Q$ is an isomorphism which preserves $i$ and such that \begin{equation*}\label{norma p}
H(dp_g(u),dp_g(u))=-4H(u,u),
\end{equation*} for all $u\in\mathcal{H}_g.$ With these notations, we have 
\begin{equation}\label{gauss p}
G=p\circ g,
\end{equation} and the map $g:M\to Spin(3,1)$ appears to be a horizontal lift to $Spin(3,1)$ of the Gauss map $G:M\to \Q.$
\end{lem}

Thus, from \eqref{gauss p}, we get \begin{equation*}
dG=dp\circ dg.
\end{equation*} Since $dp$ and $dG$ commute to the complex structures $i$ defined on $Spin(3,1), \Q$ and $M,$ so does $dg,$ and thus $g:M\to Spin(3,1)\subset \HC$ is a holomorphic map.

Using the identity \eqref{eta0}, the complex $1-$forms $\theta_1$ and $\theta_2$ are holomorphic, therefore there exists two holomorphic functions $f_1$ and $f_2$ such that
\begin{equation}\label{formas hol}
\theta_1=f_1 dz \hspace{0.3in}\mbox{and}\hspace{0.3in} \theta_2=f_2 dz,
\end{equation} where $z$ is a conformal paremeter of $(M,\mathcal{J}).$ We note that, $f_1$ and $f_2$ do not vanish simultaneously since $dG$ is assumed to be injective at every point.

The aim now is to show that the immersion $F:M\to \R^{3,1}$ induced by $\varphi$ is determined by the holomorphic functions $f_1$ and $f_2,$ and by the two smooth functions $h_1,h_2:M\to \R$ such that $$\vec{H}:=h_1 e_3+h_2 e_4,$$ the components of the mean curvature vector in the parallel frame $(e_3,e_4)$ of $E.$ We first observe that the immersion is determined by $g:M\to Spin(3,1)\subset\HC$ and by the orthonormal and parallel frame $(e_1,e_2)$ of $TM.$ \\

\begin{prop} The immersion $F:M\to \R^{3,1}$ is such that 
\begin{equation*}
dF(X)=g^{-1}\ (\omega_1(X)\ i\e+\omega_2(X)\ I)\ \widehat{g}
\end{equation*} for all $X\in TM,$ where $\omega_1,\omega_2:TM \to \R$ are the dual forms of $e_1$ and $e_2.$
\end{prop}
\begin{proof} We have \begin{equation*}
dF(X)=\laa X\cdot\varphi,\varphi\raa=g^{-1}\ [X]\ \widehat{g},
\end{equation*} where $[X]\in\HC$ stands for the coordinates of $X\in Cl(TM\oplus E)$ in $\tilde{s}.$ Recalling that $[e_1]=i\e$ and $[e_2]=I$ in $\tilde{s},$ we have $[X]=X_1 i\e + X_2 I,$ where $X_1,X_2$ are the coordinates of $X\in TM$ in $(e_1,e_2).$
\end{proof}

In the following proposition, we precise how to recover the map $g$ and the frame $(e_1,e_2)$ from the holomorphic functions $f_1$ and $f_2$ and from the smooth functions $h_1$ and $h_2:$ \\

\begin{prop}\label{recover} \textit{1-} $g$ is determined by $f_1$ and $f_2,$ up to the multiplication on the right by a constant belonging to $Spin(3,1).$ 

\noindent \textit{2-} Define $\alpha_1,\alpha_2:M \to \C$ such that 
\begin{equation*}
e_1=\alpha_1 \hspace{0.3in}\mbox{and}\hspace{0.3in} e_2=\alpha_2
\end{equation*} in the parameter $z.$ The functions $\alpha_1,\alpha_2,f_1,f_2,h_1$ and $h_2$ are linked by 
\begin{equation}\label{rela func}
(\alpha_1 i\e+\alpha_2 I)(f_1 J+f_2 iK)=(h_1 J+ h_2 K).
\end{equation} In particular, if $f_1^2-f_2^2\neq 0,$ we get 
\begin{equation}\label{rela func 2}
\alpha_1 i\e +\alpha_2 I=-(h_1 J+ h_2 K)\frac{f_1 J+ f_2 iK}{f_1^2-f_2^2},
\end{equation} that is, the frame $(e_1,e_2)$ in the coordinates $z$ is determined by $f_1,f_2,h_3$ and $h_4.$
\end{prop}
\begin{proof} \textit{1-} The solution $g$ of the equation $dg\ g^{-1}=[\eta]$ is unique, up to multiplication on the right by a constant belonging to $Spin(3,1).$ \\
\textit{2-} In $\tilde{s},$ the Dirac equation $D\varphi=\vec{H}\cdot\varphi$ is given by \begin{equation*}
-\widehat{[e_1]}[\nabla_{e_1}\varphi]+\widehat{[e_2]}[\nabla_{e_2}\varphi]=\widehat{[\vec{H}]}[\varphi];
\end{equation*} since $dg(X)=[\nabla_X\varphi]=[\eta(X)]g,$ we get \begin{equation*}
i\e\ dg(e_1)g^{-1}+I\ dg(e_2)g^{-1}=h_1 J+h_2 K,
\end{equation*} using \eqref{eta0} and \eqref{formas hol} we have $dg(e_1)g^{-1}=\alpha_1(f_1 J+f_2 iK)$ and $dg(e_2)g^{-1}=\alpha_2(f_1 J+f_2 iK)$ that implies \eqref{rela func}. Equation \eqref{rela func 2} is a consequence of \eqref{rela func}, together with the following observation: $\xi\in\HC$ is invertible if and only if $H(\xi,\xi)=\overline{\xi}\xi\neq 0;$ its inverse is then $\xi^{-1}=\frac{\overline{\xi}}{H(\xi,\xi)}.$
\end{proof}

\begin{remar} The complex numbers $\alpha_1$ and $\alpha_2,$ considered as real vector fields on $M,$ are independient and satisfy $[\alpha_1,\alpha_2]=0:$ since the metric on $M$ is flat, there is a local diffeomorphism $\psi:\R^2\to M$ such that $e_1=\frac{\partial \psi}{\partial x}$ and $e_2=\frac{\partial \psi}{\partial y}.$
\end{remar}

The interpretation of the condition $f_1^2-f_2^2=0$ is the following: using \eqref{eta0} and the identities given in Lemma \ref{lema p} we get 
\begin{equation}\label{fns}
G^*H=H(dG,dG)=-4H(dg,dg)=-4(f_1^2-f_2^2)dz^2;
\end{equation} thus, if $f_1^2-f_2^2=0$ in $x\in M,$ $dG_x(T_xM)$ belongs to the union of two complex lines through $G(x)$ in the Grassmannian $\Q;$ in particular, the osculador space in $x$ is degenerate (\textit{i.e.} the first normal space in $x$ is $1-$dimensional); a similar and more detailed description for spacelike surfaces is given in \cite[Section 6.2.3]{bayard1}.

We will gather the previous results to construct flat timelike immersions with flat normal bundle from initial data. \\

\begin{coro}\label{flat0} Let $(U,z)$ be a simply connected domain in $\C,$ and consider 
\begin{equation*}
\theta_1=f_1 dz \hspace{0.2in}\mbox{and}\hspace{0.2in} \theta_2=f_2 dz
\end{equation*} where $f_1,f_2:U\to C$ are two holomorphic functions such that  $f_1^2-f_2^2\neq 0.$  Suppose that $h_1,h_2:U\to\R$ are smooth functions such that 
\begin{equation}\label{campos0}
\alpha_1:=-i\frac{h_1f_1+h_2f_2}{f_1^2-f_2^2} \hspace{0.2in}\mbox{and}\hspace{0.2in} \alpha_2:=\frac{h_2f_1-h_1f_2}{f_1^2-f_2^2}
\end{equation} considered as real vector fields on $U,$ are independent at every point and satisfy $[\alpha_1,\alpha_2]=0.$ Then, if  $g:U \to Spin(3,1)\subset\HC$ is a map solving 
\begin{equation}\label{componente0}
dg\ g^{-1}=\theta_1 J +\theta_2\ iK,
\end{equation} and if we set 
\begin{equation}\label{xi 0}
\xi:=g^{-1}\ (\omega_1\ i\e+\omega_2\ I)\ \widehat{g}
\end{equation} where $\omega_1,\omega_2:TU\to\R$ are the dual $1-$forms of $\alpha_1,\alpha_2\in\Gamma(TU),$ the function $F=\int\xi:U\to\R^{3,1}$ defines a timelike isometric immersion with $K=K_N=0.$ 

Reciprocally, the isometric immersion of a timelike surface $M$ in $\R^{3,1}$ such that $K=K_N=0,$ with regular Gauss map and whose osculating spaces are everywhere not degenerate, are locally of this form.
\end{coro}
\begin{proof}
We consider $E=U\times \R^2$ the trivial vector bundle on $U$ and we denote by $(e_3,e_4)$ the canonical basis of $\R^2.$ Let us define $s=(e_1,e_2,e_3,e_4)$ where $e_1=\alpha_1$ and $e_2=\alpha_2$ in $\R^2\simeq\C,$ and let us consider the metric on $U$ such that $(e_1,e_2)$ is a orthonormal frame; this metric is flat and the frame $(e_1,e_2)$ is parallel since $[e_1,e_2]=0$ by hypothesis. Let $\tilde{s}$ be a section of the trivial bundle $\tilde{Q}\to U$ such that $\pi(\tilde{s})=s,$ where $\pi:\tilde{Q}=S^1_{\C}\times U\to (SO(1,1)\times SO(2))\times U$ is the natural projection. We consider $g:U \to Spin(3,1)\subset\HC$ the unique solution, up to the natural right action of $Spin(3,1),$ of the equation \eqref{componente0}:
this equation is solvable since $\eta':=\theta_1 J +\theta_2\ iK$ satisfies the structure equation $d\eta'(X,Y)-[\eta'(X),\eta'(Y)]=0,$ for all $X,Y\in\Gamma(TU).$ The definition \eqref{campos0} of $\alpha_1$ and $\alpha_2$ is equivalent to \eqref{rela func 2}, which traduces that $\varphi:=[\tilde{s},g]\in \Sigma=\tilde{Q}\times\HC/\rho$ is a solution of the Dirac equation $D\varphi=\vec{H}\cdot\varphi$ where $\vec{H}:=h_1e_3+h_2e_4$ (see the proof of Proposition \ref{recover}). Moreover, setting $\omega_1,\omega_2$ for the dual $1-$forms of $e_1,e_2\in\Gamma(TU),$  the $1-$form $\xi,$ given in \eqref{xi 0},
is such that $\xi(X)=\laa X\cdot\varphi,\varphi\raa;$ thus $\xi$ is closed, and a primitive of $\xi$ defines a timelike isometric immersion in $\R^{3,1}\subset \HC$ with induced metric $-\omega_1^2+\omega_2^2.$ Since the Gauss map of the immersion is $G=i\ g^{-1}Ig$ (Lemma \ref{gauss map1}) and since $g$ is a holomorphic map (by \eqref{componente0}), we get that $G$ is a holomorphic map, and thus that $K=K_N=0$ (Proposition \ref{gauss0}).
\end{proof}

\begin{remar}\label{flat1} A flat timelike immersion with flat normal bundle and regular Gauss map, and whose osculating spaces are everywhere not degenerate (\textit{i.e.} such that $G^*H\neq 0$ at every point), is determined by two holomorphic functions $f_1,f_2:U\to \C$ such that $f_1^2-f_2^2\neq 0$ on $U$ and by two smooth functions $h_1,h_2:U\to\R$ such that the two complex numbers $\alpha_1$ and $\alpha_2$ defined by \eqref{campos0}, considered as real vector fields, are independent at every point and such that $[\alpha_1,\alpha_2]=0$ on $U.$
\end{remar}

Considering further a holomorphic map $h:U\to\C$ such that $h^2=f_1^2-f_2^2,$ and setting $z'$ for the parameter such that $dz'=h(z)dz,$ we have 
\begin{equation*}
g^*H=H(dg.dg)=H(dgg^{-1},dgg^{-1})=(f_1^2-f_2^2)dz^2=dz'^2,
\end{equation*} and thus, in $z',$ 
\begin{equation}\label{nueva funcion}
g'g^{-1}=\cosh \psi\ J+\sinh \psi\ iK
\end{equation} for some holomorphic function $\psi:U'\to\C.$ The parameter $z'$ may be interpreted as the complex arc length of the holomorphic curve $g:U\to Spin(3,1),$ and the holomorphic function $\psi$ as the complex angle of $g'$ in the trivialization $TSpin(3,1)=Spin(3,1)\times T_1Spin(3,1).$ Observe that, from the definition \eqref{nueva funcion} of $\psi,$ the derivative $\psi'$ may be interpreted as the complex geodesic curvature of the holomorphic curve $g:U\to Spin(3,1).$ The immersion thus only depends on the single holomorphic function $\psi,$ instead of the two holomorphic functions $f_1$ and $f_2.$ Moreover, the two relations in \eqref{campos0} then simplify to 
\begin{equation}\label{campos 1}
\alpha_1=-i(h_1\cosh\psi+h_2\sinh\psi) \hspace{0.2in}\mbox{and}\hspace{0.2in}
\alpha_2=h_2\cosh\psi-h_1\sinh\psi.
\end{equation}
Note that the new parameter $z'$ may be only locally defined, since the map $z\to z'$ may be not one-to-one in general.\\

\begin{coro}\label{flat2} Let $U\subset \C$ be a simply connected domain, and let $\psi:U\to\C$ be a holomorphic function. Suppose that $h_1,h_2:U\to \R$ are smooth functions such that $\alpha_1$ and $\alpha_2,$ real vector fields defined by \eqref{campos 1}, are independet at every point and satisfy $[\alpha_1,\alpha_2]=0$ on $U.$ Then, if $g:U\to Spin(3,1)\subset \HC$ is a holomorphic map solving 
\begin{equation*}
g'g^{-1}=\cosh \psi\ J+\sinh \psi\ iK,
\end{equation*} and if we set 
\begin{equation*}
\xi:=g^{-1}\ (\omega_1\ i\e+\omega_2\ I)\ \widehat{g}
\end{equation*} where $\omega_1,\omega_2:TU\to\R$ are the dual $1-$forms of $\alpha_1,\alpha_2\in\Gamma(TU),$ the function $F=\int\xi:U\to\R^{3,1}$ defines a timelike isometric immersion with $K=K_N=0.$ 

Reciprocally, the isometric immersion of a timelike surface $M$ in $\R^{3,1}$ such that $K=K_N=0,$ with regular Gauss map and whose osculating spaces are everywhere not degenerate, are locally of this form.
\end{coro}

\section{Flat timelike surfaces in the De Sitter space}\label{app-4}
In this section, using spinors we deduce a result of Aledo, G\'alvez and Mira given in \cite[Corollary 5.1]{AGM} concerning the conformal representation of a flat timelike surface in three-dimensional De Sitter space.

Keeping the notation of Section \ref{preliminares}, we consider the isomorphism of algebras 
\begin{eqnarray*}
A: \HC & \longrightarrow & M_2(\C) \\
q=q_1 \e + q_2 I + q_3 J + q_4 K & \longmapsto & A(q)=\begin{pmatrix} q_1+iq_2 & q_3+iq_4 \\ -q_3+iq_4 & q_1-iq_2 \end{pmatrix}.
\end{eqnarray*} We note the following properties: 
\begin{equation}\label{prop iso}
A(\widehat{\overline{q}})=A(q)^* \hspace{0.2in}\mbox{and}\hspace{0.2in} H(q,q)=\det(A(q))
\end{equation} for all $q\in \HC,$ where $A(q)^*$ is the conjugate transpose of $A(q).$ Using \eqref{prop iso}, we get an identification \begin{equation*}
\R^{3,1}=\{ \xi\in\HC \mid \widehat{\overline{\xi}}=-\xi\}\simeq iHerm(2),
\end{equation*} where the metric $\la\cdot,\cdot\ra$ of $\R^{3,1}$ identifies with $\det$ defined on $iHerm(2);$ moreover, the De Sitter space $\mathbb{S}^{2,1}\subset \R^{3,1}$ (defined in \eqref{de sitter}) is described as \begin{equation*}
\mathbb{S}^{2,1}=\left\lbrace B\begin{pmatrix} 0 & i \\ i & 0 \end{pmatrix}B^*\ \mid\ B\in Sl_2(\C) \right\rbrace\ \subset\ iHerm(2).
\end{equation*}

\begin{coro}\label{flat3} Let $M$ be a Riemann surface, $B:M \to Sl_2(\C)$ be a holomorphic map such that there exists $\theta,\omega$ nowhere vanishing holomorphic $1-$forms that satisfy $$B^{-1}dB=\begin{pmatrix} 0 & \theta \\ \omega & 0 \end{pmatrix}.$$ Assume moreover that $\Im m(\frac{\omega}{\theta})\neq 0.$ Then $$F:=B\begin{pmatrix} 0 & i \\ i & 0 \end{pmatrix}B^*:M \longrightarrow \mathbb{S}^{2,1}$$ defines, with the induced metric, a flat timelike isometric immersion. 

Conversely, an isometric immersion of a simply connected flat timelike surface $M$ in the De Sitter space  may be described as above. 
\end{coro}
\begin{proof} The proof of the direct statement is obtained by a direct computation; see \cite{AGM}. We thus only prove the converse statement. We suppose that there exists a flat isometric immersion $F:M\to \mathbb{S}^{2,1}$ of a simply connected timelike surface $M.$ Using the natural isometric embedding $\mathbb{S}^{2,1}\hookrightarrow \R^{3,1},$ we get a flat timelike immersion $M \hookrightarrow \R^{3,1} $ with flat normal bundle and regular Gauss map, and we can consider the complex structure $\mathcal{J}$ on $M$ such that its Gauss map is holomorphic (see Section \ref{gauss section}). We denote by $E$ its normal bundle, $\vec{H}\in\Gamma(E)$ its mean curvature vector field and $\Sigma:=M\times \HC$ the spinor bundle of $\R^{3,1}$ restricted to $M.$ The immersion $F$ is given by 
\begin{equation*}
F=\int\xi \hspace{0.2in}\mbox{where}\hspace{0.2in} \xi(X)=\laa X\cdot\varphi,\varphi\raa,
\end{equation*} for some spinor field $\varphi\in\Gamma(\Sigma)$ solution of $D\varphi=\vec{H}\cdot\varphi$ and such that $H(\varphi,\varphi)=1$ (the spinor field $\varphi$ is the restriction to $M$ of the constant spinor field $+\e$ or $-\e\in\HC$ of $\R^{3,1}$). Using Proposition \ref{inm reduc}, we have 
\begin{equation}\label{inm flat}
F=\laa e_4\cdot\varphi,\varphi\raa
\end{equation} where $e_4\in\Gamma(E)$ is normal to $\mathbb{S}^{2,1}$ in $\R^{3,1}.$ We choose a parallel frame $\tilde{s}\in\Gamma(\tilde{Q})$ adapted to $e_4,$ \textit{i.e.} such that $e_4$ is the fourth vector of $\pi(\tilde{s})\in\Gamma(Q_1\times_M Q_2):$ in $\tilde{s},$ equation \eqref{inm flat} reads 
\begin{equation}\label{inm flat1}
F=\overline{[\varphi]}K\widehat{[\varphi]} \simeq A(\overline{[\varphi]}K\widehat{[\varphi]})=A(\overline{[\varphi]})A(K)A(\widehat{[\varphi]})
\end{equation} where $[\varphi]\in\HC$ represents $\varphi$ in $\tilde{s}.$ Thus, setting $B:=A(\overline{[\varphi]})$ and using \eqref{prop iso} we have that $B$ belongs to $Sl_2(\C)$ (since $H(\varphi,\varphi)=1$) and $B^*=A(\widehat{[\varphi]}).$ From \eqref{inm flat1} we thus get $$F\simeq B\begin{pmatrix} 0 & i \\ i & 0 \end{pmatrix} B^*.$$ With respect to the complex structure induced on $M$ (by the Gauss map), $B:M\to Sl_2(\C)$ is a holomorphic map (since $[\varphi]:M\to Spin(3,1)\subset \HC$ is a holomorphic map and $A$ is $\C-$linear). Note that $dB=A(d\overline{[\varphi]}),$ using \eqref{eta0} we obtain 
\begin{align*}
B^{-1}dB&=A([\varphi]\ d\overline{[\varphi]})=-A(d[\varphi]\ \overline{[\varphi]}) \\
&= -A(\theta_1 J+\theta_2 iK)=\begin{pmatrix} 0 & -\theta_1+\theta_2 \\ \theta_1+\theta_2 \end{pmatrix},
\end{align*} where $\theta_1+\theta_2=:\omega$ and $-\theta_1+\theta_2=:\theta$ are holomorphic  $1-$forms (formula \eqref{eta0}). We also note that $\omega$ and $\theta$ nowhere vanish: if we suppose that $\omega=0$ or $\theta=0$ in $x\in M,$ using \eqref{formas hol} we get $0=\omega\theta=-\theta_1^2+\theta_2^2=-(f_1^2-f_2^2)dz^2,$ thus, from \eqref{fns} we obtain $G^*H=0$ in $x,$ in particular, the first normal space in $x$ is $1-$dimensional which is not possible since $G$ is regular; see \cite[Lemma 2.2]{AGM}. Finally, it is not difficult to verify that  $dF$ injective reads $\Im m(\frac{\omega}{\theta})\neq 0.$
\end{proof}

\paragraph*{Acknowledgements.} The author was supported by the project CONACyT 265667.


\begin{thebibliography}{10}
\bibitem{AGM} J.A. Aledo, J.A. G\'alvez, P. Mira, \emph{Isometric immersions of $\mathbb{L}^2$ into $\mathbb{L}^4,$} Diff. Geom. Appl. \textbf{24} (2006) 613-627.

\bibitem{BB} C. Baikoussis, D.E. Blair, \emph{On the Gauss map of ruled surfaces,} Glasgow Math. J. \textbf{34} (1992) 355-359.

\bibitem{bar}C. B\"{a}r, \emph{Extrinsec bounds for the eigenvalues of the Dirac operator,} Ann. Glob. Anal. Geom. \textbf{16} (1998) 573-596.

\bibitem{bayard1}P. Bayard, \emph{On the spinorial representation of spacelike surfaces into 4-dimensional Minkowski space,} J. Geom. Phys. \textbf{74} (2013) 289-313.

\bibitem{bayard_lawn_roth}P. Bayard, M.A. Lawn, J. Roth, \emph{Spinorial representation of surfaces into 4-dimensional space forms,} Ann. Glob. Anal. Geom. \textbf{44:4} (2013) 433-453.


\bibitem{bayard_patty}P. Bayard, V. Patty, \emph{Spinor representation of Lorentzian surfaces in $\R^{2,2},$} J. Geom. Phys. \textbf{95} (2015) 74-95.


\bibitem{friedrich}Th. Friedrich, \emph{On the spinor representation of surfaces in Euclidean 3-space,} J. Geom. Phys. \textbf{28} (1998) 143-157.

\bibitem{lawn}M.A. Lawn, \emph{A spinorial representation for Lorentzian surfaces in $\R^{2,1}$,} J. Geom. Phys. \textbf{58:6} (2008) 683-700.

\bibitem{lawn_roth}M.A. Lawn, J. Roth, \emph{Spinorial characterization of surfaces in pseudo-Riemannian space forms,} Math. Phys. Anal. and Geom. \textbf{14:3} (2011) 185-195.

\bibitem{morel}B. Morel, \emph{Surfaces in $\mathbb{S}^3$ and $\mathbb{H}^3$ via spinors,} Actes du s\'eminaire de th\'eorie spectrale, Institut Fourier, Grenoble \textbf{23} (2005) 9-22.

\bibitem{roth}J. Roth, \emph{Spinorial characterizations of surfaces into 3-homogeneous manifolds,} J. Geom. Phys. \textbf{60} (2010) 1045-1061.

\bibitem{oneill} B. O'neill, \emph{Semi-Riemannian Geometry with applications to relativity,} Pure and applied mathematics, 1983.

\bibitem{patty}V. Patty, \emph{A generalized Weierstrass representation of Lorentzian surfaces in $\R^{2,2}$ and applications,} Int. J. Geom. Methods Mod. Phys. \textbf{13} (2016) 1650074 (26 pages).
\end{thebibliography}
\end{document}